\documentclass[12pt,a4]{amsart}
\usepackage{amsfonts}
\usepackage{mathrsfs}
\usepackage{amsthm,amsxtra}
\usepackage{amssymb}
\usepackage{amsmath}
\usepackage{amscd}
\usepackage{enumerate}
\usepackage{hyperref}
\usepackage{enumitem}
\usepackage{calc}
\usepackage{t1enc}
\usepackage[mathscr]{eucal}
\usepackage{indentfirst}
\usepackage{graphicx}
\graphicspath{ {images/} }
\usepackage[margin=2.9cm]{geometry}
\usepackage{tikz-cd}
\usepackage{tikz}

\usepackage{ragged2e}

\makeatletter
\newtheorem*{rep@theorem}{\rep@title}
\newcommand{\newreptheorem}[2]{%
\newenvironment{rep#1}[1]{%
 \def\rep@title{#2 \ref{##1}}%
 \begin{rep@theorem}}%
 {\end{rep@theorem}}}
\makeatother

\theoremstyle{plain}

\newtheorem{introthm}{Theorem}
\newtheorem{introcoro}[introthm]{Corollary}
\newtheorem{theorem}{Theorem}[section]
\newtheorem{proposition}[theorem]{Proposition}
\newtheorem{lemma}[theorem]{Lemma}
\newtheorem{corollary}[theorem]{Corollary}
\newreptheorem{theorem}{Theorem}
\newreptheorem{lemma}{Lemma}
\newreptheorem{proposition}{Proposition}
\newreptheorem{corollary}{Corollary}

\theoremstyle{definition}

\newtheorem{example}[theorem]{Example}

\newtheorem{question}[theorem]{Question}

\theoremstyle{remark}
\newtheorem{remark}[theorem]{Remark}

\DeclareMathOperator{\rank}{rank}

\DeclareMathOperator{\st}{st}
\DeclareMathOperator{\lk}{lk}
\DeclareMathOperator{\supp}{supp}
\DeclareMathOperator{\esupp}{esupp}

\DeclareMathOperator{\Deck}{Deck}
\DeclareMathOperator{\Inn}{Inn}
\DeclareMathOperator{\GL}{GL}
\DeclareMathOperator{\Aut}{Aut}
\DeclareMathOperator{\PAut}{PAut}

\DeclareMathOperator{\Out}{Out}
\DeclareMathOperator{\FAut}{FAut}
\DeclareMathOperator{\FOut}{FOut}
\DeclareMathOperator{\LAut}{LAut}
\DeclareMathOperator{\LOut}{LOut}

\DeclareMathOperator{\LPAut}{LPAut}

\DeclareMathOperator{\IA}{IA}
\DeclareMathOperator{\FD}{FDeck}
\DeclareMathOperator{\C}{C}
\DeclareMathOperator{\SLAut}{SLAut}
\DeclareMathOperator{\SAut}{SAut}
\DeclareMathOperator{\F}{\mathbb{F}}
\DeclareMathOperator{\Z}{\mathbb{Z}}

\makeatletter
\DeclareFontFamily{OMX}{MnSymbolE}{}
\DeclareSymbolFont{MnLargeSymbols}{OMX}{MnSymbolE}{m}{n}
\SetSymbolFont{MnLargeSymbols}{bold}{OMX}{MnSymbolE}{b}{n}
\DeclareFontShape{OMX}{MnSymbolE}{m}{n}{
    <-6>  MnSymbolE5
   <6-7>  MnSymbolE6
   <7-8>  MnSymbolE7
   <8-9>  MnSymbolE8
   <9-10> MnSymbolE9
  <10-12> MnSymbolE10
  <12->   MnSymbolE12
}{}
\DeclareFontShape{OMX}{MnSymbolE}{b}{n}{
    <-6>  MnSymbolE-Bold5
   <6-7>  MnSymbolE-Bold6
   <7-8>  MnSymbolE-Bold7
   <8-9>  MnSymbolE-Bold8
   <9-10> MnSymbolE-Bold9
  <10-12> MnSymbolE-Bold10
  <12->   MnSymbolE-Bold12
}{}

\let\llangle\@undefined
\let\rrangle\@undefined
\DeclareMathDelimiter{\llangle}{\mathopen}%
                     {MnLargeSymbols}{'164}{MnLargeSymbols}{'164}
\DeclareMathDelimiter{\rrangle}{\mathclose}%
                     {MnLargeSymbols}{'171}{MnLargeSymbols}{'171}
\makeatother

\begin{document}

\title{Liftable automorphisms of right-angled Artin groups}

\author{Sangrok Oh}
\address{Sangrok Oh, University of the Basque Country \\ Department of
Mathematics \\ Bilbao, Spain}
\email{sangrokoh.math@gmail.com}
\urladdr{https://sites.google.com/view/sangrokohmath}
	
\author{Donggyun Seo}
\address{Donggyun Seo, Seoul National University \\ Department of Mathematical Sciences \\ Seoul, Korea}
\email{seodonggyun@snu.ac.kr}
\urladdr{sites.google.com/view/donggyunseo/}

\author{Philippe Tranchida}
\address{Philippe Tranchida, Universit\'e Libre de Bruxelles, D\'epartement de Math\'ematique, C.P.216 - Alg\`ebre et Combinatoire, Boulevard du Triomphe, 1050 Brussels, Belgium.}
\email{tranchida.philippe@gmail.com}
\urladdr{sites.google.com/view/ptranchi/}


\begin{abstract}
Given a regular covering map $\varphi:\Lambda \to \Gamma$ of graphs, we investigate the subgroup $\LAut(\varphi)$ of the automorphism group $\Aut(A_\Gamma)$ of the right-angled Artin group $A_\Gamma$. This subgroup comprises all automorphisms that can be lifted to automorphisms of $A_\Lambda$. We first show that $\LAut(\varphi)$ is generated by a finite subset of Laurence's elementary automorphisms. 

For the subgroup $\FAut(\varphi)$ of $\Aut(A_\Lambda)$, which consists of lifts of automorphisms in $\LAut(\varphi)$, there exists a natural homomorphism $\FAut(\varphi)\to\LAut(\varphi)$ induced by $\varphi$.
We then show that the kernel of this homomorphism is virtually a subgroup of the Torelli subgroup $\IA(A_\Lambda)$ and deduce a short exact sequence reminiscent of results from the Birman--Hilden theory for surfaces.
\end{abstract}

\maketitle

\section{Introduction}
A right-angled Artin group $A_\Gamma$ (or briefly, RAAG) is uniquely determined by its defining simplicial graph $\Gamma$ \cite{Droms}; its standard generators and relators correspond to the vertices and edges of $\Gamma$ (see Section \ref{sec:pre}).
To explore the properties of RAAGs and their (outer) automorphism groups, numerous instances have arisen, prompting investigations into the outcomes yielded by naturally induced surjective homomorphisms $\phi:A_\Lambda\to A_\Gamma$.
To cite only a few:

\begin{enumerate}[labelindent=0pt,label=(\arabic*),itemindent=1em,leftmargin=1.5em]
\item\label{Item:BB}  (When $\Gamma$ is a single vertex.) In \cite{BB97}, Bestvina and Brady showed that the kernel of $\phi$ is of type $F_n$ (a finiteness property, see \cite{BB97} for the precise definition) if and only if the flag completion of $\Gamma$ is $(n-1)$-connected. 
This fact tells us that $A_\Gamma$ can have subgroups which have complicated structures.
\item\label{Item:Torelli} 
(When $\Gamma$ is a complete graph with $|V\Gamma|=|V\Lambda|$.)
From a natural homomorphism $\Phi:\Aut(A_\Gamma)\to\GL(n,\Z)$ induced from $\phi$, one can define the Torelli subgroup $\IA(A_\Gamma)$ of $\Aut(A_\Gamma)$ as the kernel of $\Phi$.
As an analog of the Torelli subgroup of a mapping class group, the kernel of the symplectic representation of the mapping class group, the group $\IA(A_\Gamma)$ has received a fair share of attention but its structure has remained rather opaque for now. Nonetheless, its finite generation is proved in \cite{Day09},\cite{Wade12}.
\item\label{Item:CV}  (When $\Gamma$ is a specific subgraph of $\Lambda$.) In \cite{CV09} and \cite{CV11}, Charney and Vogtmann defined virtually surjective homomorphisms from $\Out(A_\Gamma)$ to $\Out(A_\Lambda)$, called \textit{restriction} and \textit{projection} maps. Using these maps, they investigated various properties of $\Out(A_\Gamma)$ including the Tits alternative, residual finiteness and finite virtual cohomological dimension.
\end{enumerate}

In the above cases, indeed, the `natural' surjective homomorphisms $\phi:A_\Lambda\to A_\Gamma$  comes from a graph morphism $\varphi:\Lambda\to\Gamma$ whose restriction to the vertex set of $\Lambda$ is surjective onto the vertex set of $\Gamma$.
In this case, an automorphism $f$ of $A_\Gamma$ is said to be \emph{liftable} if there exists an automorphism $F$ of $A_\Lambda$, called a \emph{lift} of $f$, such that $f\circ\phi = \phi\circ F$. The subset of $\Aut(A_\Gamma)$ which consists of liftable automorphisms forms a subgroup, denoted by $\LAut(\varphi)$.
An automorphism of $\Aut(A_\Lambda)$ which is a lift of an automorphism of $\Aut(A_\Gamma)$ is said to be \emph{fiber-preserving}, and the subset of $\Aut(A_\Lambda)$ consisting of fiber-preserving automorphisms also forms a group, denoted by $\FAut(\varphi)$.
By simple algebra, one can then obtain a surjective homomorphism $\Phi_I:\Inn(A_\Lambda)\twoheadrightarrow \Inn(A_\Gamma)$ (see Lemma~\ref{lem:inner}) and its extension $\Phi:\FAut(\varphi)\twoheadrightarrow\LAut(\varphi)$ such that the following diagram holds.
\[
\begin{tikzcd}
  1 \arrow[r] & \ker\Phi  \arrow[r, "f"] & \FAut(\varphi)  \arrow[r, "\Phi"] & \LAut(\varphi)  \arrow[r] & 1 \\
  1 \arrow[r] & \ker\Phi_I \arrow[u, " "] \arrow[r, "f'"] & \Inn(A_\Lambda) \arrow[u, " "] \arrow[r, "\Phi_I"] & \Inn(A_\Gamma) \arrow[u, " "] \ar[r] & 1
\end{tikzcd}
\]

In this paper, we focus our attention on the case when $\Lambda$ is a regular covering of $\Gamma$ without isolated vertices. 
Our results can be mainly divided into two parts: 
The first one is that $\LAut(\varphi)$ is generated by finitely many elements, which can be determined from $\varphi$, and the second one is that the kernel of $\Phi$ turns out to be commensurable to a subgroup of $\IA(A_\Lambda)$. Here are the precise statements.

\begin{introthm}[Theorem~\ref{thm:liftable_conjugating}, Theorem~\ref{thm:lifable_group}] \label{introthm:generation_laut}
Let $\varphi:\Lambda\to\Gamma$ be a regular covering map of graphs without isolated vertices.
Then $\LAut(\varphi)$ is generated by inversions, liftable graph symmetries, liftable transvections, and liftable partial conjugations.
In particular, $\LAut(\varphi)$ is finitely generated.

Moreover, any liftable automorphism which maps each standard generator $v$ of $A_\Gamma$ to a conjugate of $v$ is generated by liftable partial conjugations.
\end{introthm}

\begin{introthm}[Theorem~\ref{CommensurableTorelli}] \label{introthm:lift_identity}
Let $\varphi:\Lambda\to\Gamma$ be given as in Theorem~\ref{introthm:generation_laut}.
Then the group of lifts of the identity of $\Aut(A_\Gamma)$ is commensurable to a subgroup of the Torelli subgroup $\IA(A_\Lambda)$ of $\Aut(A_\Lambda)$.
\end{introthm}

The non-existence of isolated vertices is crucial in the above two results since we may lose the control of the number of lifts of graph symmetries if there are isolated vertices.
Roughly speaking, if both $\Gamma$ and $\Lambda$ are connected, for the identity graph symmetry $\sigma$ of $\Gamma$ and a vertex $u$ in $\Lambda$, there is a unique graph symmetry $\mu$ of $\Lambda$ such that $\mu$ fixes $u$ and $\varphi\circ\mu=\sigma\circ\varphi$.
On the other hand, if $\Gamma$ consists of only isolated vertices, then there are at least two such graph symmetries $\mu$ of $\Lambda$.
We defer further discussion about the non-existence condition to Section \ref{sec:isolatedvertices}. 

Here is a more detailed overview of the content of this paper.

\subsection{Finite generation}\label{subsec:FG}
Laurence \cite{MR1356145} proved that $\Aut(A_\Gamma)$ is finitely generated by identifying an elementary finite generating set consisting of graph symmetries, inversions, transvections, and partial conjugations: a (left) transvection is denoted by $T_{v'}^v: v' \mapsto vv'$ 
and a partial conjugation is denoted by $P_C^v$ where $C$ is a union of components of $\Gamma \setminus \st(v)$ 
See Section \ref{subsec:automorphisms_RAAGs} for more precise definitions.
Laurence also found that partial conjugations generate the conjugating automorphism group defined in page \pageref{def:conjugating}.

Let $\varphi:\Lambda\to\Gamma$ be a regular covering map of graphs (possibly with isolated vertices). 
In order to prove Theorem~\ref{introthm:generation_laut}, we first investigate which standard generators of $\Aut(A_\Gamma)$ are liftable.
Liftability of inversions is quite straightforward, and indeed it does not even depend on the regularity of the covering map (See Lemma \ref{lem:lift_sym_inv}). 
In contrast, transvections, partial conjugations and graph symmetries are not always liftable, and the criterions for liftability highly depend on the covering map.

\begin{introthm}[Theorem~\ref{prop:elementary_lift}, Corollary~\ref{lem:lift_transvection2}, Proposition~\ref{lift_partial_2}] \label{introthm:criteria_liftable}
Let $\varphi: \Lambda \to \Gamma$ be a regular covering map of graphs. Then the following hold:
\begin{enumerate}

\item \label{enuma:criteria_liftable} A transvection $T_{v'}^v$ is liftable if $\lk(u') \subseteq \st(u)$ for some $u \in \varphi^{-1}(v)$ and $u' \in \varphi^{-1}(v')$. In addition, the converse holds when $\Gamma$ has no isolated vertex.
\item \label{enumb:criteria_liftable} Let $v \in \Gamma$ be a vertex, and let $C$ be a union of components of $\Gamma \setminus \st(v)$. Then the partial conjugation $P_C^v$ is liftable if $\bar{C} = C$ (see Section~\ref{subsec:lift_part_conj} for the definition $\bar{C}$).
And the converse holds when $\Gamma$ has no isolated vertex.
\item A graph symmetry $\sigma$, considered as an element of $\Aut(A_\Gamma)$, is liftable if and only if it can be lifted through the covering map $\varphi$.
\end{enumerate}
See Section~\ref{sec:pre} for the definitions of $\lk(u)$ and $\st(u)$.
\end{introthm}

An isomorphism $f$ of $A_\Gamma$ is said to be \emph{essential} if for any vertex $v\in\Gamma$, a cyclically reduced word for $f(v)$ contains $v$ as a letter (see Section~\ref{sec:pre} for the definition of a cyclically reduced word).
If $\Gamma$ has no isolated vertices, then the covering map $\varphi$ restricts our choice of a lift of a liftable essential automorphism in the following sense. 

\begin{repcorollary}{lem:deck_lift}
Let $\varphi:\Lambda\rightarrow\Gamma$ be a regular covering map of graphs without isolated vertex.
Let $f\in\Aut(\Gamma)$ be a liftable automorphism which is essential. For any lift $F$ of $f$, then, there exists a deck-transformation $\mu$ such that $F\mu$ is a lift of $f$ which is essential.
\end{repcorollary}

Indeed, what Laurence proved is exactly the statement which is obtained in the statement of Theorem~\ref{introthm:generation_laut} by replacing $\LAut(\varphi)$ to $\Aut(A_\Gamma)$ and removing the word `liftable'. While following Laurence's proof, by carefully following the restrction given by Corollary~\ref{lem:deck_lift}, we first show that a conjugating automorphism which is liftable is a product of liftable elementary automorphisms, which is the latter statement of Theorem~\ref{introthm:generation_laut}. 
And then we prove the finite generation of $\LAut(\varphi)$ using the previous result, which completes the proof of Theorem~\ref{introthm:generation_laut}.

\subsection{Birman--Hilden theory}
Our original motivation for the study of the short exact sequence $1\to\ker\Phi\to\FAut(\varphi)\to \LAut(\varphi)\to 1$ was the famous results of Birman and Hilden. In \cite{MR325959}, they studied regular branched coverings of surfaces to find a concrete presentation of the mapping class group of genus two from the presentation of braid groups.
A key result of Birman--Hilden's work is the existence of a short exact sequence asserting that the quotient of the symmetric mapping class group of a surface by the deck transformation group of the covering is isomorphic to the liftable mapping class group if a hyperelliptic covering is given.
They then show in \cite{MR321071} that such a short exact sequence exists for every regular unbranched surface covering.
The precise statements can be found in Margalit--Winarski's survey article \cite[Proposition 3.1]{MR4275077}.

Note that a regular covering map $\varphi:\Lambda\rightarrow\Gamma$ of graphs induces a branched covering map between the 2-skeletons of the Salvetti complexes associated to $\Lambda$ and $\Gamma$.
With that analogy in mind, we may wish to obtain a short exact sequence from $\varphi$ similar to the one coming from Birman--Hilden theory. In general, however, there are lifts of the identity which have infinite order, which make things more complicated; see Figure \ref{CommutatorTransvection}.
Nevertheless, in the spirit of the Birman--Hilden theory applied to $\Aut(A_\Gamma)$, we are able to show the following corollary of Theorem \ref{introthm:lift_identity}:

\begin{introcoro}\label{introcoro:SES}
Let $\Gamma$, $\Lambda$ and $\varphi$ be given as in Theorem~\ref{introthm:generation_laut}.
Let $H_\Gamma: \Out(A_\Gamma) \to \operatorname{GL}(n, \mathbb{Z})$ be the homological representation obtained by abelianization, and similarly $H_\Lambda$ for $\Lambda$.
Then we have the following short exact sequence
$$
1 \to \Deck(\varphi) \to H_\Lambda(\FOut(\varphi)) \to H_\Gamma(\LOut(\varphi)) \to 1.
$$
\end{introcoro}

This short exact sequence is somehow analogous to the short exact sequence obtained from the Birman-Hilden theory, even though it is clearly weaker as it only holds on the level of homology.

We remark that the kernel of the projection $\FOut(\varphi) \to \LOut(\varphi)$ is much bigger than $\Deck(\varphi)$ so that we cannot remove the homological representations from the above short exact sequence.
This also support our choice of terminology ``fiber-preserving (outer) automorphism group'' instead of ``symmetric (outer) automorphism group''. From this perspective, a symmetric outer automorphism $F$ would be an automorphism of $A_\Lambda$ which commutes with $\Deck(\varphi)$. Under these definitions, the symmetric automorphism group can be shown to be a proper subgroup of the fiber-preserving outer automorphism group.
In this paper, the theory is developed under the setting of automorphism groups, not of outer automorphism groups, but the results still hold after passing to the quotient by the inner automorphism group.

\begin{remark}
The Salvetti complex is a classifying space of a right-angled Artin group (RAAG).
Every branched covering map of Salvetti complexes with a single branched point is induced from an unbranched covering of graphs.
But the converse may not hold. Namely, there exists a covering of graphs which does not induce a branched covering of Salvetti complexes.
In Figure \ref{CommutatorTransvection} of page \pageref{CommutatorTransvection}, the $2$-fold covering map sends a hexagon to a triangle.
Although the Salvetti complex associated to the triangle has a $3$-dimensional cube, the Salvetti complex associated to the hexagon does not.
\end{remark}


\subsection{Comparison with relative outer automorphism groups}
Day--Wade \cite{MR4072157} developed a beautiful theory for an infinite family of subgroups of outer automorphism groups of RAAGs, so called \emph{relative automorphism groups}.
This family includes $\Out(A_\Gamma)$ itself, and each member admits a normal series whose factors have simple forms.
Similar to Theorem \ref{introthm:generation_laut}, Day--Wade \cite[Theorem D]{MR4072157} said that every relative automorphism group is generated by a subset of Laurence's generating set.

A subgroup of a RAAG $A_\Gamma$ is said to be a \emph{special subgroup} if it is generated by an induced subgraph of $\Gamma$.
For two collections $\mathcal{G}$ and $\mathcal{H}$ of special subgroups of $A_\Gamma$, the \emph{relative outer automorphism group} $\Out^0(A_\Gamma; \mathcal{G}, \mathcal{H}^t)$ is the group of outer automorphisms that preserve the conjugacy classes of all $G \in \mathcal{G}$ and fix the elements of all $H \in \mathcal{H}$.
Theorem D in \cite{MR4072157} says that the set of all inversions, transvections and partial conjugations in $\Out^0(A_\Gamma; \mathcal{G}, \mathcal{H}^t)$ generates $\Out^0(A_\Gamma; \mathcal{G}, \mathcal{H}^t)$ itself.

That being said, it seems hard to establish a concrete relationship between the collection of liftable (outer) automorphism groups and the collection of relative (outer) automorphism groups.
Proposition 3.9 in \cite{MR4072157} deduces that some inversions may not be contained in a relative outer automorphism group. But every liftable outer automorphism group contains all inversions by Lemma \ref{lem:lift_sym_inv} in this paper.
In fact, if a relative outer automorphism group $\Out^0(A_\Gamma; \mathcal{G}, \mathcal{H}^t)$ is a liftable automorphism group for some regular covering map of graphs, then $\mathcal{H}$ must be empty.



\subsection{Graphs with isolated vertices}\label{sec:isolatedvertices} 
Here we highlight some distinct behaviours of covering maps of graphs with and without isolated vertices.

\begin{example}\label{ex:freegroups}
For a graph $\Gamma=\{a,b,c\}$ consisting of $3$ isolated vertices, let 
\[\Lambda=\{a_1,a_2,a_{3},b_{1},b_2,b_{3},c_1,c_2,c_{3}\}\] 
be a regular $3$-covering of $\Gamma$ such that the covering map $\varphi:\Lambda\to\Gamma$ sends $a_{i}$, $b_i$ and $c_i$ to $a$, $b$ and $c$, respectively. 
Then any automorphism of $\Aut(\F_3)$ is liftable, i.e. there exists a surjective homomorphism $\Phi:\FAut(\F_{9})\to\Aut(\F_3)$.

The kernel $\ker \Phi$ of $\Phi$ obviously contains the subgroup isomorphic to $\IA(\F_3)\times\IA(\F_3)\times\IA(\F_3)$. 
It also contains the subgroup which is isomorphic to $\ker\phi$ for the homomorphism $\phi:\F_9\to\F_3$ induced from $\varphi$. However, there are elements of $\IA(\F_9)$ which are not contained in $\ker \Phi$.
For instance, the inner automorphism by $b_1a_3a_2^{-1}b_2^{-1}$ in $\Aut(A_\Lambda)$ is in $\ker\Phi$ but the inner automorphism by $b_1a_3a_2^{-1}c_1b_2^{-1}c_1^{-1}$ is not. 
We could not obtain any relation between $\ker\Phi$ and $\IA(A_\Lambda)$, while it can be known if there are no isolated vertices in $\Gamma$ (and thus $\Lambda$), as in Theorem~\ref{introthm:lift_identity}.
\end{example}

If $\Gamma$ has no isolated vertices, as seen in Section~\ref{subsec:FG}, we can say that liftability of automorphisms is restricted by the graph structure of $\Gamma$ and the covering map $\varphi$.
On the other hand, if $\Gamma$ has at least two isolated vertices, then there is an example in which the restriction no longer holds (Example~\ref{Ex:FreeGroupGraphSymmetry}).
Thus, we finish this section with several questions.

\begin{question}
Let $\varphi:\Lambda\to\Gamma$ be a regular covering map of graphs (possibly, having isolated vertices).
\begin{enumerate}
\item Can we find a necessary condition for an automorphism to be liftable with respect to $\varphi$?
\item Can we characterize the kernel of $\Phi:\FAut(\varphi)\to\LAut(\varphi)$? Or, can we obtain any relation between the kernel and the Torelli subgroup $\IA(A_\Lambda)$?
\end{enumerate}
Even the case when $\Gamma$ consists of $m$ isolated vertices would be interesting as $\FAut(\varphi)$ and $\ker\Phi$ are subgroups of $\Out(\F_{mn})$ where $n$ is the degree of the covering map.
\end{question}

\subsection{Guide to readers}
Section \ref{sec:pre} introduces Servatius and Laurence's works about right-angled Artin groups and their automorphism groups.

In Section \ref{sec:lift_element}, we explain what kinds of elementary automorphisms are liftable. Section \ref{subsec:simplicial} consists of elementary facts about surjective homomorphisms induced from regular covering maps of graphs. We can see the proof of a sufficient condition for the liftability of a transvection in Section \ref{subsec:lift_inv_transv} and a partial conjugation in Section \ref{subsec:lift_part_conj}.

Section \ref{sec:lift_isom} is technical but contains the essential lemmas of our paper. Here it becomes apparent why we need to consider only graphs without isolated vertices. Theorem \ref{thm:lift_iso} is a generalization of Laurence's work \cite{MR1356145} that explains one of the behavior of liftable graph symmetries. 

Based on the results in Section~\ref{sec:lift_isom}, the content of Section \ref{sec:liftable_conjugating} is the proof of the finite generation of the group of all liftable conjugating automorphisms. And then we prove Theorem~\ref{introthm:generation_laut} in Section \ref{sec:gen_liftable}.
In Section \ref{sec:lift_identity}, we investigate the lifts of the identity and prove Theorem \ref{introthm:lift_identity}.
Finally, in Section \ref{sec:lift_transv}, we complete the discussion about liftability of transvections, which finishes the proof of Theorem \ref{introthm:criteria_liftable}.

\subsection{Acknowledgement}
We would like to thank Sang-hyun Kim, Hyungryul Baik, Thomas Koberda, Junseok Kim, Richard Wade and Dan Margalit for useful comments.
The first and second authors were supported by the National Research Foundation of Korea (NRF) grant No. 2021R1C1C200593811 from the Korea government(MSIT). The third author is partially supported by Samsung Science and Technology Foundation under Project Number SSTF-BA1702-01.

%
%
%
%

\section{Preliminaries} \label{sec:pre}
In this paper, a graph $\Gamma$ is a finite 1-dimensional simplicial complex, and its vertex and edge sets are denoted by $V\Gamma$ and $E\Gamma$, respectively. The cardinality of $V\Gamma$ is denoted by $|\Gamma|$.
By a subgraph $\Gamma'$, we always mean an induced one, i.e., $\Gamma'$ is the subgraph induced by $V\Gamma'\subset V\Gamma$. Moreover, we denote the subgraph induced by $V\Gamma\setminus V\Gamma'$ by $\Gamma\setminus\Gamma'$.
Lastly, any maps between graphs are assumed to be simplicial.

Associated to $\Gamma$ is the \emph{right-angled Artin group} (RAAG) $A_\Gamma$ which has the following presentation:
\[\langle v \in V\Gamma \mid [u, v]  = 1 \text{ for all }\{u, v\} \in E\Gamma  \rangle.\]
A \textit{word} $w$ in the alphabet $V\Gamma$ is a concatenation of elements in $V\Gamma \cup (V\Gamma)^{-1}$ and the \emph{word length} of $w$ is the number of elements appreating in the concatenation. 
Any element of $A_\Gamma$ corresponds to an equivalence class of words in the alphabet $V\Gamma$; two words are \textit{equivalent} if they correspond to the same element of $A_\Gamma$. 
With that in mind, we say that an element $g$ in $A_\Gamma$ is \emph{represented} by a word $w$ (in the alphabet $V\Gamma$) if $w$ is one of the words in the equivalence class corresponding to $g$. 

A word $w$ is said to be \textit{reduced} if it has minimal word length in its equivalence class. 
Among all words representing conjugates of $g$, one having minimal word length is said to be a \textit{cyclically reduced word} for $g$. If $g$ can be represented itself by a cyclically reduced word, then $g$ is said to be \textit{cyclically reduced}.
The \textit{support} and the \emph{essential support} of $g$, denoted by $\supp(g)$ and $\esupp(g)$, are the sets of vertices appearing in reduced and cyclically reduced words for $g$, respectively (these two sets are well-defined for elements in RAAGs). 

Let $v\in\Gamma$ be a vertex. The degree of $v$, denoted by $\deg_\Gamma v$, is the number of vertices which are adjacent to $v$. The link of $v$, denoted by $\lk_\Gamma(v)$, is the subgraph of $\Gamma$ induced by vertices adjacent to $v$, and the star of $v$, denoted by $\st_\Gamma(v)$, is the subgraph of $\Gamma$ induced by $\lk_\Gamma(v)$ and $v$. In particular, $|\lk_\Gamma(v)|=\deg_\Gamma v$ and $|\st_\Gamma(v)|=\deg_\Gamma v+1$. If there is no confusion, we omit the underlying graph when talking about degrees, links and stars.

Using links and stars, we can define a preorder $\lesssim$, called \emph{link-star order}, on $V\Gamma$ by $$u \lesssim v ~\text{if and only if}~ \lk(u) \subseteq \st(v).$$
We write $V\Gamma_{\gtrsim v}$ for the set of vertices $v_i$ which satisfy $v_i\gtrsim v$.
Two vertices $u$ and $v$ are said to be \textit{equivalent} if $u \lesssim v$ and $v \lesssim u$, and the equivalence class of $v$ is denoted by $[v]$.
The following are two well-known facts about equivalence classes which will be used frequently in this paper.
\begin{enumerate}
    \item The equivalence class $[v]$ induces either a complete subgraph or a totally disconnected subgraph.
    \item For a graph symmetry $\sigma$ of $\Gamma$, $[v]$ and $[\sigma(v)]$ have the same cardinality, and $[v]$ is complete if and only if $[\sigma(v)]$ is complete.
\end{enumerate}

\begin{lemma}\label{lemma:adjacent}
For two vertices $v_1$ and $v_2$ in $\Gamma$, if they are adjacent, then any vertex $u_1$ in $[v_1]$ is adjacent to any vertex $u_2$ in $[v_2]$ unless $u_1=u_2$.
\end{lemma}
\begin{proof}
If $v_1$ and $v_2$ are equivalent, then $[v_1](=[v_2])$ induces a complete subgraph, and thus, the lemma holds.
Otherwise, since $v_2\in\lk(v_1)$ and $v_2\notin [v_1]$, we have $v_2\in\lk(u_1)$, i.e., $u_1\in\lk(v_2)$. 
Then $\lk(u_2)$ contains $u_1$ since $u_1$ is not contained in $[v_2]=[u_2]$.
\end{proof}

Sometimes, it is convenient to order the vertices of $\Gamma$ in a way that is coherent with the link-star order.
\begin{lemma}\label{OrderOnV}
The vertices of $\Gamma$ can be ordered in such a way that $v_i \gnsim v_j$ implies $i < j$.
\end{lemma}
\begin{proof}
This can be done by first gathering all the maximal vertices in $V\Gamma$. 
After ignoring the previously chosen vertices, we choose maximal vertices among the remaining ones and repeat this process until we are done. Then the order on vertices we have constructed satisfies the lemma due to the following fact: if $v\gnsim u$ and $u\gnsim w$, then $v\gnsim w$.
\end{proof}

For an element $g\in A_\Gamma$, let $x$ be a cyclically reduced word representing $g$. The link of $g$, denoted by $\lk(g)$ (or $\lk(x)$), is the subgraph of $\Gamma$ induced by vertices which are adjacent to every vertex in $\supp(x)$ (in particular, any vertex in $\lk(g)$ is not contained in $\supp(x)$).
In RAAGs, centralizers of cyclically reduced words are well described using the links of the words.

\begin{theorem}[Centralizer theorem \cite{MR1023285}] \label{thm:centralizer}
Each cyclically reduced word $x$ in $A_\Gamma$ can be written as $x_1^{r_1} \dots x_k^{r_k}$ such that the centralizer of $x$, denoted by $C(x)$, is $$C(x) = \langle x_1 \rangle \times \dots \times \langle x_k \rangle \times \langle \lk(x) \rangle.$$
\end{theorem}

The centralizer theorem implies that $x_1, \cdots, x_k$ commute pairwise and have pairwise disjoint supports.
For each $g \in A_\Gamma$, let $\operatorname{rank}(g)$ denote the rank of the first homology of $C(g)$.
As a consequence of the centralizer theorem, we have $\rank(g) = k + | \lk(x) |$ where $x$ is a cyclically reduced word representing $g$ and $k$ comes from $x$ and the above theorem. In particular, $\rank(v)=|\st(v)|$ for any vertex $v\in\Gamma$.

\begin{lemma}[Proposition 3.5 \cite{MR1356145}] \label{lem:rank}
Let $x$ be a cyclically reduced word in $A_\Gamma$ with a decomposition $x = x_1^{r_1} \dots x_k^{r_k}$ given in Theorem~\ref{thm:centralizer}.
Then the following hold.
\begin{enumerate}
\item \label{enum1:rank} For every vertex $v \in \supp x$, we have $|\st(v)| \geq \rank(x)$.
\item \label{enum2:rank} If $|\st(v)| = \rank(x)$ for some $v \in \supp x$, then 
\begin{enumerate}
\item \label{enuma:rank} $\st(v) \cap \supp x_1 = \{ v \}$,
\item \label{enumb:rank} $x_i$ is a vertex which commutes with $x$ for all $i>1$, and
\item \label{enumc:rank} $v \lesssim v'$ for every $v' \in \supp x$.
\end{enumerate}
\end{enumerate}
\end{lemma}

\begin{proof}
\noindent\eqref{enum1:rank}
Let us decompose $C(x)$ into $\langle x_1 \rangle \times \dots \times \langle x_k \rangle \times \langle \lk(x) \rangle$ by Theorem \ref{thm:centralizer}.
Without loss of generality, suppose $v$ is contained in $\supp x_1$.
Because $\lk(v) \supseteq \supp x_2 \cup \dots \cup \supp x_k \cup \lk(x)$, we have $| \st(v) | = 1 + | \lk(v) | \geq k+|\lk(x)| =\rank(x) $.
\vspace{1mm}

\noindent\eqref{enum2:rank}
As above, let $C(x) = \langle x_1 \rangle \times \dots \times \langle x_k \rangle \times \langle \lk(x) \rangle$, and suppose $v$ is contained in $\supp x_1$.
Due to the equality, for all $i > 1$, $\supp x_i$ has only one vertex so the equation $x_i = v_i^{r_i}$ holds for some $v_i \in \lk(v)$ and some nonzero integer $r_i$.
Since $|\st(v)| = \rank(x)$, the vertex set $V\lk(v)$ of $\lk(v)$ must be $\{ v_2, \dots, v_k \} \cup \lk(x)$ and thus any vertex in $\supp x_1 \setminus \{ v \}$ does not commute with $v$.
For any $v' \in \supp x$, because $v'$ commutes with any vertex in $\lk(v)$ (possibly, $v'\in\lk(v)$), we have $v' \gtrsim v$.
\end{proof}

Every map $\varphi : \Lambda \to \Gamma$ of graphs induces a homomorphism $\phi: A_\Lambda \to A_\Gamma$ sending $v$ to $\varphi(v)$ for all $v \in V\Lambda$; if $\varphi$ is surjective (injective, resp.), $\phi$ is surjective (injective, resp.).
If $\varphi$ is surjective, any inner automorphisms of $A_\Gamma$ are inherited from those of $A_\Lambda$ by $\phi$, which is derived from the following basic fact.

\begin{lemma} \label{lem:inner}
Any surjective homomorphism $\phi:G\to H$ between two groups induces a surjective homomorphism
$\Phi_I: \Inn(G) \to \Inn(H)$ satisfying that $\Phi_I(\iota)(\phi(x)) = \phi(\iota(x))$ for all $\iota \in \Inn(G)$ and $x \in G$.
\end{lemma}

\begin{proof}
Define a map $\Phi_I: \Inn(G) \to \Inn(H)$ by $$\Phi_I(\iota_g) = \iota_{\phi(g)}\quad \textrm{for}\quad g\in G$$ where $\iota_g$ ($\iota_{\phi(g)}$, resp.) is an inner automorphism of $G$ ($H$, resp.) sending an element $x$ ($y$, resp.) to $gxg^{-1}$ ($\phi(g)y(\phi(g))^{-1}$, resp.).
The map $\Phi_I$ is obviously a surjective homomorphism since $\phi$ is a surjective homomorphism. 
\end{proof}

\subsection{Automorphisms of RAAGs} \label{subsec:automorphisms_RAAGs}

Among all automorphisms of $A_\Gamma$, some of them, used as building blocks, have been called \textit{elementary}. They are of four possible types, listed here: 

\begin{description}
\item[Graph symmetries] A graph symmetry $\sigma$ of $\Gamma$ can be extended to an automorphism of $A_\Gamma$ that permutes vertices by $\sigma$.
\item[Inversions] An inversion is an automorphism mapping each vertex $v$ to either $v$ or $v^{-1}$.
\item[Transvections] For two vertices $v,v'$ satisfying $v \lesssim v'$, the (left) transvection $T^v_{v'}$ maps $v$ to $vv'$ and fixes all other vertices.
\item[Partial conjugations] For a vertex $v$ and a union $C$ of components of $\Gamma \setminus \st(v)$, the partial conjugation $P_C^v$ is the automorphism sending all vertices $v_i$ of $C$ to $v v_i v^{-1}$ and fixing all other vertices. 
\end{description}

\begin{remark}\label{rem:convention_partial}
If $C_1, \cdots, C_n$ are distinct components of $\Gamma\setminus \st(v)$, the product $P_{C_1}^v\cdots P_{C_n}^v$ is sometimes called an \emph{extended partial conjugation} as in \cite{MR4072157}. In this paper, for the sake of simplicity, by partial conjugations we mean extended ones.
\end{remark}

Laurence proved the conjecture of Serviatus that $\Aut(A_\Gamma)$ is finitely generated by those four types of automorphisms.

\begin{theorem}[\cite{MR1356145}] \label{thm:Laurence_generation}
The automorphism group  of any right-angled Artin group  is generated by graph symmetries, inversions, transvections, and partial conjugations.
\end{theorem}

In the rest of this subsection, we recall the main ingredients that were used by Laurence in the proof of Theorem \ref{thm:Laurence_generation}, since we will follow a similar approach in this paper.

Essential supports of elements in $A_\Gamma$ play a crucial role, and the next theorem allows us to control, for an automorphism $f \in \Aut(A_\Gamma)$, the essential support of $f(v)$ for all $v\in V\Gamma$.

\begin{theorem}[Lemma 6.2 in \cite{MR1356145}] \label{thm:iso_thm}
For any automorphism $f \in \Aut(A_\Gamma)$, there exists a graph symmetry $\sigma$ such that $\sigma(v)$ is contained in $\esupp f(v)$ for every vertex $v \in \Gamma$.
\end{theorem}

For convenience, we say that an automorphism $f\in\Aut(A_\Gamma)$ is \textit{essential} if $\esupp f(v)$ contains $v$ for every $v\in V\Gamma$.
In particular, $f\sigma^{-1}$ is essential in the above theorem.

In Theorem \ref{thm:iso_thm}, since $\rank(f(v))=|\st(v)|=| \st(\sigma(v))|$, by Lemma \ref{lem:rank}\eqref{enumc:rank}, we have $\esupp f(v)\subseteq V\Gamma_{\gtrsim\sigma(v)}$ and the following fact. 

\begin{proposition} \label{prop:essupp_auto}
Let $f \in \Aut(A_\Gamma)$ be an automorphism. For any vertex $v \in V\Gamma$ and $v'\in\esupp f(v)$, $\esupp f(v)\subseteq V\Gamma_{\gtrsim v'}$ if and only if $|\st(v)| = |\st(v')|$. In particular, the vertices in $\esupp f(v)$ which are minimal under the link-star order in $\Gamma$ are all equivalent.
\end{proposition}

\label{def:conjugating}An automorphism $f$ of $A_\Gamma$ is said to be \emph{conjugating} if $f(v)$ is a conjugate of $v$ for each $v \in \Gamma$, and the set of conjugating automorphisms of $A_\Gamma$ forms a subgroup $\PAut(A_\Gamma)$ of $\Aut(A_\Gamma)$. 
Laurence observed that $\PAut(A_\Gamma)$ is the subgroup generated by partial conjugations.

\begin{theorem}[Theorem 2.2 in \cite{MR1356145}] \label{thm:conjugating}
Any conjugating automorphism of a RAAG can be represented by the product of partial conjugations.
\end{theorem}
In Section \ref{sec:liftable_conjugating}, we briefly introduce Laurence's proof of Theorem \ref{thm:conjugating}, that serves an algorithm to find a word of partial conjugations representing a given conjugating automorphism.

In Laurence's proof, we can obtain a decomposition of an automorphism of $A_\Gamma$.

\begin{proposition} \label{prop:decom}
Let $\mathcal{H}$ be the subgroup of $\Aut(A_\Gamma)$ generated by graph symmetries, inversions, and transvections.
Then for every $f \in \Aut(A_\Gamma)$, there exists a decomposition $f = gh$ such that
\begin{enumerate}
\item \label{enuma:decom} $g \in \PAut(A_\Gamma)$, and
\item \label{enumc:decom} $h \in \mathcal{H}$ and $h(v)$ is cyclically reduced for every $v \in V\Gamma$.
\end{enumerate}
\end{proposition}

\begin{proof}
By Theorem \ref{thm:iso_thm}, there is a graph symmetry $\sigma$ such that each $v \in V\Gamma$ is contained in $\esupp f \sigma(v)$.
By Lemma \ref{OrderOnV}, the vertices of $\Gamma$ are labelled by $v_1, \dots, v_k$ such that $v_i \gnsim v_j$ implies $i < j$.
We claim that for each $i \in \{1, \dots, k\}$, there exists a product of transvections and inversions, denoted by $h_i$, such that 
\begin{itemize}
    \item $h_i (v_j)$ is cyclically reduced and $f\sigma h_i(v_j)$ is conjugate to $v_j$ for all $j \leq i$, and 
    \item $h_i (v_j)=v_j$ for all $j> i$.
\end{itemize}
For convenience, let $h_0$ be the identity map, and let $v_0$ be the identity of $A_\Gamma$.
Fix $i \in \{1, \dots, k\}$, and assume that such $h_{i-1}$ exists. 

Let $\iota_i$ be an inner automorphism such that $\iota_i f\sigma  h_{i -1} (v_i)$ is cyclically reduced. Since $\supp(\iota_i f\sigma h_{i -1} (v_i))$ contains $v_i$, by Proposition \ref{prop:essupp_auto}, $\supp(\iota_i f \sigma  h_{i-1}(v_i))\subset V\Gamma_{\gtrsim v_i}$.
So there exists a product of transvections and inversions, denoted by $T_i$, such that $T_i$ fixes all vertices other than $v_i$ and $ T_i(v_i) = (\iota_i f\sigma h_{i-1})^{-1} (v_i)$.
For each $j \leq i$, the element $f \sigma h_{i-1} T_i (v_j)$ is conjugate to $v_j$.
Then $h_i := h_{i - 1} T_i$ is what we want, so the claim holds.

In conclusion, $h := (\sigma h_k)^{-1}$ satisfies Item \eqref{enumc:decom}, and $g := f \sigma h_k$ is a conjugating automorphism by the above claim.
Therefore, the statement holds.
\end{proof}

%
%
%
%

\section{Liftable elementary automorphisms} \label{sec:lift_element}
\textbf{From now on, $\varphi:\Lambda\to\Gamma$ is a covering map of graphs with the group of deck-transformations $\Deck(\varphi)$, and $\phi:A_\Lambda\rightarrow A_\Gamma$ is the homomorphism induced from $\varphi$ until the end of this paper.}

We say $\varphi$ is \emph{regular} if the induced action of $\Deck(\varphi)$, on each fiber is transitive, i.e., for $u_1,u_2\in V\Lambda$, if $\varphi(u_1)=\varphi(u_2)$, there is a deck transformation $\sigma\in\Deck(\varphi)$ such that $\sigma(u_1)=u_2$.
Unless otherwise stated, $\varphi$ is assumed to be regular.

In this section, we find sufficient conditions for inversions, transvections and partial conjugations of $\Aut(A_\Gamma)$ to be liftable. More precisely, for each such automorphism $f$ of $A_\Gamma$, we decide whether there exists an automorphism $F$ of $A_\Lambda$ such that $f\phi = \phi F$.

\subsection{Regular covering of a graph}\label{subsec:simplicial} 
We start with two elementary facts which hold even when $\varphi$ is not regular:
\begin{itemize}
    \item If $\varphi(u_1)=\varphi(u_2)$, then $d(u_1,u_2)\geq 3$ (we consider that $d(u_1,u_2)$ is $\infty$ when $u_1$ and $u_2$ are not in the same component of $\Lambda$).
    \item Isolated vertices map to isolated vertices.
\end{itemize}

The covering map $\varphi$ induces a suborder $\lesssim_\varphi$ on $V\Gamma$ as follows: for two vertices $v,v'\in V\Gamma$, we say $v \lesssim_\varphi v'$ if for every $u \in \varphi^{-1}(v)$, there exists a vertex $u' \in \varphi^{-1}(v')$ such that $\lk(u) \subseteq \st(u')$, and $v\sim_{\varphi}v'$ (or $v$ and $v'$ are \textit{$\varphi$-equivalent}) if $v_1\lesssim_\varphi v_2$ and $v_2\lesssim_\varphi v_1$.
Similarly for $V\Gamma_{\gtrsim v}$, we write $V\Gamma_{\gtrsim_\varphi v}$ for the set of vertices $v_i$ which satisfy $v_i\gtrsim_\varphi v$.

In general, however, the suborder $\lesssim_\varphi$ is not equal to the link-star order $\lesssim$ of $\Gamma$.
For instance, in Figure \ref{fig:2_fold_simplicial_cover}, the two vertices $u$ and $v$ in the base graph are comparable, that is $v \lesssim u$, but any $u_i\in\varphi^{-1}(u)$ and $v_j \in \varphi^{-1}(v)$ cannot be compared with the link-star order in $\Lambda$ so that $u$ and $v$ are not comparable under the suborder $\lesssim_{\varphi}$.

\begin{figure}
\centering
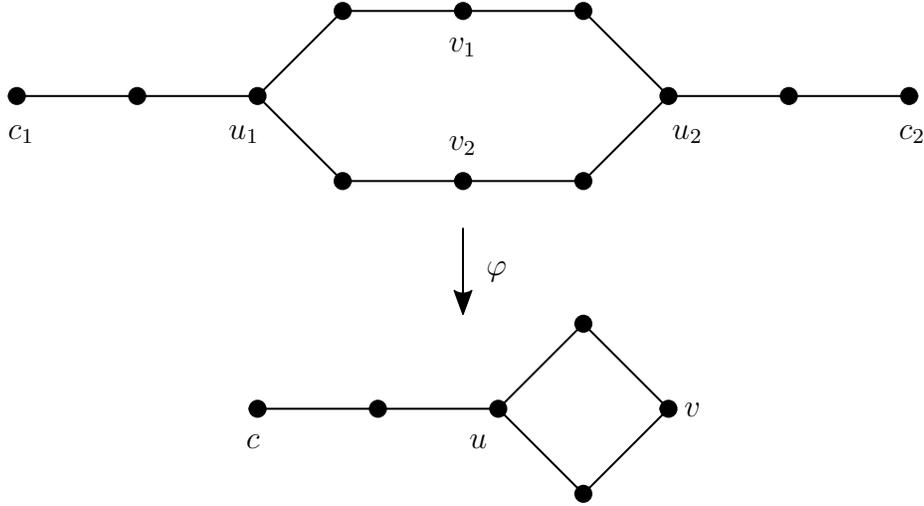
\caption{$v\lesssim u$ but $u$ and $v$ are not comparable under $\lesssim_{\varphi}$.}
\label{fig:2_fold_simplicial_cover}
\end{figure}

\begin{lemma} \label{lem:cover-preserving}
The link-star order is preserved by $\varphi$, i.e., for any $u_1,u_2\in V\Lambda$, if we have $u_1\lesssim u_2$, then we have $\varphi(u_1)\lesssim\varphi(u_2)$. Moreover, for any $u\in V\Lambda$, we have $\varphi(V\Lambda_{\gtrsim u})=V\Gamma_{\gtrsim_\varphi \varphi(u)}$.  
\end{lemma}
\begin{proof}
Since $\varphi$ is locally an isomorphism, $\varphi$ induces a bijection between the vertices in $\lk(u_i)$ ($\st(u_i)$, resp.), and the vertices in $\lk(\varphi(u_i))$ ($\st(\varphi(u_i))$, resp.).
For $u'\in \lk(u_1)\subseteq\st(u_2)$, moreover, $\varphi(u')$ is contained in both $\lk(\varphi(u_1))$ and $\st(\varphi(u_2))$. 
By combining these facts, we have $\lk(\varphi(u_1))\subseteq\st(\varphi(u_2))$ and therefore, we have $\varphi(u_1) \lesssim \varphi(u_2)$, proving the first statement.

By the previous paragraph, for $u'\in V\Lambda_{\gtrsim u}$, we have $\varphi(u')\gtrsim_\varphi \varphi(u)$, and thus $\varphi(V\Lambda_{\gtrsim u})\subset V\Gamma_{\gtrsim_\varphi \varphi(u)}$. 
Conversely, for a vertex $v \in V\Gamma_{\gtrsim_\varphi \varphi(u)}$, by definition, there exists $u' \in \varphi^{-1}(v)$ such that $u \lesssim u'$, which implies that $v$ is contained in $\varphi(V\Lambda_{\gtrsim u})$. Therefore, we have $\varphi(V\Lambda_{\gtrsim u})=V\Gamma_{\gtrsim_\varphi \varphi(u)}$, concluding the proof.
\end{proof}

We note that even though there is a bijection between the vertex set of $\lk(u)$ and the vertex set of $\lk(\varphi(u))$, $\varphi(\lk(u))$ may not coincide with $\lk(\varphi(u))$ in general since we define the link of a vertex as an induced subgraph, and similarly for $\st(u)$. For instance, in Figure \ref{figure:NotLift}, $\lk(v_1)$ is the disjoint union of a vertex and an edge but $\lk(u_1)$ consists of three isolated vertices.

\begin{lemma} \label{lem:simplicial_criterion}
Suppose two vertices $u, u' \in \Lambda$ satisfy $u \lesssim u'$ and $\varphi(u) =\varphi(u')$. 
If $u$ is not an isolated vertex, we have $u=u'$. Otherwise, $u'$ is also an isolated vertex.
\end{lemma}
\begin{proof}
If $u$ is equal to $u'$, the proof is done. 

Suppose $u$ and $u'$ are distinct.
The conditions that $\varphi$ is a covering map of graphs and that $\varphi(u)=\varphi(u')$ imply that either $u$ and $u'$ are in two distinct components of $\Lambda$, or $d(u,u')\geq 3$.
The condition $u\lesssim u'$ implies that either $u$ is an isolated vertex or $d(u,u')\leq 2$.
Combining these two facts with the fact that the fiber of an isolated vertex consists of isolated vertices only, we can finish the proof.
\end{proof}

\begin{lemma}\label{EquivalentClass}
For a vertex $v\in V\Gamma$, let $[v]_\varphi$ be the $\varphi$-equivalent class of $v$. Then the subgraph induced by $[v]_\varphi$ is either complete or edgeless.
\end{lemma}
\begin{proof}
Since $v_1\lesssim_\varphi v_2$ implies $v_1\lesssim v_2$, any $\varphi$-equivalence class is contained in some equivalence class induced from the link-star order. Since the latter class induces either a complete subgraph or a totally disconnected subgraph, the lemma holds.
\end{proof}

\subsection{Liftable inversions and transvections}\label{subsec:lift_inv_transv}
The case of inversions is the easiest. Indeed, they are all liftable. 

\begin{lemma} \label{lem:lift_sym_inv}
All inversions of $\Aut(A_\Gamma)$ are liftable. This statement even holds when $\varphi$ is not regular.
\end{lemma}

\begin{proof}
Let $f_v\in \Aut(A_\Gamma)$ be the inversion sending the vertex $v$ of $\Gamma$ to $v^{-1}$ (and fixing all other vertices).
For each $u_i\in\varphi^{-1}(v)$, let $F_{u_i}\in\Aut(A_\Lambda)$ be the inversion sending the vertex $u_i$.
Then the product of inversions $$\tilde{f}_v(u):=\prod_i F_{u_i} = \begin{cases} u^{-1} & \text{if }u \in \varphi^{-1}(v), \\ u,&\text{otherwise,} \end{cases}$$ is a lift of $f_v$.
\end{proof}

On the other hand, not all transvections on $A_\Gamma$ are liftable. Instead, one can find a sufficient condition for the liftability of transvections by using the suborder $\lesssim_\varphi$.

\begin{lemma} \label{lem:lift_transvection}
For two vertices $v,v'\in V\Gamma$ satisfying $v\lesssim v'$, the transvection $T_{v}^{v'}\in \Aut(A_\Gamma)$ is liftable if $v \lesssim_\varphi v'$.
\end{lemma}

\begin{proof}
Let $u_1,\cdots,u_n$ be the collection of vertices in $\varphi^{-1}(v)$.
By definition, for each $u_i$, there exists a unique vertex $u'_i\in \varphi^{-1}(v')$ such that $u_i\lesssim u'_i$ and hence the transvection $T_{u_i}^{u'_i}$ is a well-defined automorphism of $A_\Lambda$.
Let $$\widetilde{T} := \prod_{i} T_{u_i}^{u'_i}$$ be an element in $\Aut(A_\Lambda)$. 
Then we have $\phi\circ\widetilde{T}(u_i) = \phi(u_i u'_i) = v v' = T_v^{v'}(\phi(u_i))$ for each $u_i$, and $\phi\circ\widetilde{T}(u) = \phi(u) = T_v^{v'}(\phi(u))$ for other vertices $u$.
Therefore, $\widetilde{T}$ is a lift of $T_{v}^{v'}$.
\end{proof}

In Section \ref{sec:lift_transv}, we show that the converse of Lemma \ref{lem:lift_transvection} also holds (Corollary \ref{lem:lift_transvection2}) if $\Gamma$ and $\Lambda$ have no isolated vertices.

\subsection{Liftable partial conjugations} \label{subsec:lift_part_conj}
Certainly, relying on Lemma~\ref{lem:inner}, we can deduce that specific inner automorphisms of $A_\Gamma$ are liftable; in fact, every inner automorphism is liftable.

\begin{lemma} \label{lem:inner_liftable}
Every inner automorphism $\iota$ of $A_\Gamma$ is liftable.
More precisely, there exists an inner automorphism $\tilde\iota$ of $A_\Lambda$ such that $\phi\circ\tilde\iota(x) = \iota\circ\phi(x)$ for all $x \in A_\Lambda$.
\end{lemma}
\begin{proof}
Let $y \in A_\Gamma$ be an element satisfying $\iota(y') = yy'y^{-1}$ for all $y' \in A_\Gamma$.
Choose $z \in \phi^{-1}(y)$, and let $\tilde\iota$ denote the inner automorphism conjugating elements by $z$.
Then for all $x \in A_\Lambda$, we have $\phi\circ\tilde\iota(x) = \phi(zxz^{-1}) = y\phi(x)y^{-1} = \iota\circ\phi(x)$.
\end{proof}

In general, however, partial conjugations may not be liftable. In order to write down a sufficient condition for a partial conjugation to be liftable, we need the definition of a kind of partition of the set of components of a vertex inducing partial conjugations.

For each vertex $v \in \Gamma$ ($u\in \Lambda$, resp.), let $\mathcal{C}(v)$ ($\mathcal{C}(u)$, resp.) be the collection of components of $\Gamma \setminus \st(v)$ ($\Lambda\setminus \st(u)$, resp.). 
For a connected subgraph $\widetilde{A}$ of $\Lambda$ not intersecting $\st(u)$, let $C(u,\widetilde{A})$ be the component in $\mathcal{C}(u)$ containing $\widetilde{A}$. 
For $B \in \mathcal{C}(v)$ and a component $\widetilde{B}$ of $\varphi^{-1}(B)$, we define a subgraph $\bar{B}$ as the image of the intersection of the $C(u,\widetilde{B})$'s for all $u \in \varphi^{-1}(v)$ under $\varphi$, i.e., $\bar{B} = \varphi(D)$ where 
$$D=\bigcap_{u \in \varphi^{-1}(v)} C(u,\widetilde{B}).$$
Indeed, $\bar{B}$ does not depend on the choice of $\widetilde{B}$ since $\varphi$ is regular. 

We remark that $\bar{B}$ is a union of elements in $\mathcal{C}(v)$ containing $B$.
If $B'\in\mathcal{C}(v)$ intersects $\bar{B}$, then there exists a component $\widetilde{B}'$ of $\varphi^{-1}(B')$ such that $\widetilde{B}'$ and $D$ intersect.
For each $u \in \varphi^{-1}(v)$, since $\widetilde{B}'$ intersects $C(u,\widetilde{B})$ but not $\st(u)$, we have $\widetilde{B}'\subseteq C(u,\widetilde{B}')=C(u,\widetilde{B})$.
Thus, we have $\widetilde{B}' \subseteq \bigcap_{u\in\varphi^{-1}(v)} C(u,\widetilde{B}')=D$, 
which implies that $B' \subseteq \bar{B'}=\bar{B}$.
In particular, for any $B_1,B_2\in\mathcal{C}(v)$, either $\bar{B_1}=\bar{B_2}$ or $\bar{B_1}\cap\bar{B_2}=\emptyset$.

We say two disjoint subgraphs are \emph{adjacent} if there is an edge such that one endpoint is contained in one subgraph and the other endpoint is contained in the other subgraph.  

\begin{lemma}\label{lem:Minimality}
Suppose $\Gamma$, $v$, $B$, $\widetilde{B}$, $C(u,\widetilde{B})$'s for each $u\in\varphi^{-1}(v)$ and $D$ are given as above. 
If $\Gamma$ and $\Lambda$ are connected, then either $\varphi^{-1}(B)\subseteq D$ or $C(u,\widetilde{B})=D$ for a vertex $u\in\varphi^{-1}(v)$ whose star is adjacent to $\widetilde{B}$.
\end{lemma}
\begin{proof}
If there is a path from a vertex in $\widetilde{B}$ to a vertex in $\Lambda\setminus\widetilde{B}$, then the path must pass through $\st(u)$ for some $u\in\varphi^{-1}(v)$.
If $u$ is the only vertex in $\varphi^{-1}(v)$ whose star is adjacent to $\widetilde{B}$, then $D=C(u,\widetilde{B})$.

Suppose $u_1,\cdots,u_n$ $(n\geq2)$ are vertices in $\varphi^{-1}(v)$ whose stars are adjacent to $\widetilde{B}$. Note that $C(u_i,\widetilde{B})$ contains $\st(u_j)$ for any $j\neq i$.
Let $b$ be a vertex in $\widetilde{B}$ and $\sigma$ a non-trivial deck transformation in $\Deck(\varphi)$ with order $k$.
Since $\varphi$ is regular, in order to finish the proof, it suffices to show that $C(u_i,\widetilde{B})$ contains $\sigma(\widetilde{B})$.

Since $\Lambda$ is assumed to be connected, there exists an induced path $\ell$ from $b$ to $\sigma(b)\in\sigma(\widetilde{B})$ satisfying the following property: 
\begin{itemize}
    \setlength{\itemindent}{0em}
    \item[$(\clubsuit)$] for a component $A$ of $\varphi^{-1}(\st(v))$ or $\varphi^{-1}(B')$ for $B'\in\mathcal{C}(v)$, $\ell\cap A$ is either empty or connected.
\end{itemize}
If $\ell$ passes through $\st(u_i)$ and $\st(u_j)$ for $i\neq j$, then one can modify $\ell$ to be an induced path, satisfying the property $(\clubsuit)$, which passes through exactly one of $\st(u_i)$ and $\st(u_j)$; this is possible since $\st(u_i)$ and $\st(u_j)$ are disjoint.
Repeating this process, one can have an induced path $\ell$, satisfying the property $(\clubsuit)$, from $b$ to $\sigma(b)$ which passes through only one of the $\st(u_i)$'s, say $\st(u_1)$.
Similarly, one can further modify $\ell$ to be an induced path, satisfying the property $(\clubsuit)$ that passes through only one of $\st(\sigma(u_i))$'s, say $\st(\sigma(u_j))$ (possibly, $u_1=\sigma(u_j)$).
If $u_i\neq u_1$, then $C(u_i,\widetilde{B})$ contains $\sigma(\widetilde{B})$ due to the existence of $\ell$.
Otherwise, since $\sigma^{k-1}(u_j)$ is one of $u_i$'s which is not $u_1$, the existence of the path $\sigma(\ell)\cup\cdots\cup\sigma^{k-1}(\ell)$ implies that $C(u_i,\widetilde{B})$ contains $\sigma(\widetilde{B})$.
\end{proof}

In the above lemma, it can easily be deduced that if $\varphi^{-1}(B)\subseteq D$ and there is $B_1\in\mathcal{C}(v)$ contained in $\bar{B}$, then $\varphi^{-1}({B}_1)\subseteq D$.
In particular, if $\st(u)$ is adjacent to $\widetilde{B}$ for some $u\in\varphi^{-1}(v)$, then there is a component $\widetilde{B}_1$ of $\varphi^{-1}(B_1)$ such that $\widetilde{B}_1$ is adjacent to $\st(u)$ and $C(u,\widetilde{B})$ contains $\widetilde{B}_1$.


\begin{proposition} \label{prop:partial_lift}
Let $\varphi:\Lambda\to\Gamma$ be a regular covering map of graphs.
For every vertex $v \in \Gamma$ and $B \in \mathcal{C}(v)$, the partial conjugation $P_{\bar{B}}^v$ is liftable. 
\end{proposition}
\begin{proof}
If $B$ is not contained in a component of $\Gamma$ containing $v$, then 
we have $\bar{B}=B$ and thus the partial conjugation $\widetilde{P}_{\varphi^{-1}(B)}^u$ for any vertex $u\in\varphi^{-1}(v)$ is a lift of $P_{\bar{B}}^v$. 
So, we assume that $\Gamma$ is connected.

We first see the case when $\Lambda$ is connected. Let $\varphi^{-1}(v)=\{u_i\mid i\in I\}$ and let $\widetilde{B}$ be a component of $\varphi^{-1}(B)$ and $C_i=C(u_i,\widetilde{B})$.
Let $D=\bigcap_{i\in I}{C_i}$.
By Lemma~\ref{lem:Minimality}, among $\st(u_i)$'s, every one is adjacent to $D$, or only one is.
In the latter case, $D=C(u_i,\widetilde{B})$ for some $i\in I$ and thus $D\cap\sigma(D)=\emptyset$ for any deck transformation $\sigma$. Then, 
$$\prod_{\sigma\in\Deck(\varphi)}\sigma\widetilde{P}_{D}^{u_i}\sigma^{-1}=\prod_{\sigma\in\Deck(\varphi)}\widetilde{P}_{C(\sigma(u_i),\sigma(\widetilde{B}))}^{\sigma(u_i)}$$ 
is a lift of $P_{\bar{B}}^v$. 

Suppose $D$ is adjacent to all the $\st(u_i)$'s. 
Let $C\in \mathcal{C}(u_k)\setminus\{C_k\}$. Then $C$ does not contain any of $\st(u_i)$'s.
Since $C_i$ for $i\neq k$ contains $\st(u_k)$, it also contains $C$.
Thus $C$ is contained in $(C_1 \cap \dots \cap \hat{C}_k \cap \dots \cap C_n)=\bigcap_{i\neq k}C_i$.
Let $E_i = \Lambda \setminus C_i$.
Then we have $E_i = (C_1 \cap \dots \cap \hat{C}_i \cap \dots \cap C_n) \setminus D$. It means that $\Lambda$ is partitioned into $\{ D\}\cup\{E_i\mid i\in I \}$.
Then the product $\widetilde{P} = \widetilde{P}_{C_1}^{u_1} \dots \widetilde{P}_{C_n}^{u_n}\in \Aut(A_\Lambda)$ has the following rules: 
$$\widetilde{P}(w) = \begin{cases} u_1 \dots u_n w u_n^{-1} \dots u_1^{-1} & \text{if} ~ w \in D ~ \text{and} \\ u_1 \dots \hat{u}_i \dots u_n w u_n^{-1} \dots \hat{u}_i^{-1} \dots u_1^{-1} & \text{if} ~ w \in E_i, \end{cases}$$ 
and
$$\phi\circ \widetilde{P}(w) = \begin{cases} v^{n} \varphi(w) v^{-n} & \text{if} ~ w \in \varphi^{-1}(\bar{B}) ~ \text{and} \\ v^{(n-1)} \varphi(w) v^{-(n-1)}, & \text{otherwise.} \end{cases}$$
For the inner automorphism $\iota$ by $v$, then the above equation deduces the equation $\phi\circ \widetilde{P} = P_{\bar{B}}^v \iota^{(k-1)}\circ \phi$, i.e., $P_{\bar{B}}^v \iota^{(k-1)}\in\Aut(A_\Gamma)$ is liftable.
Therefore, by Lemma \ref{lem:inner_liftable}, $P_{\bar{B}}^v$ is liftable.

Finally, suppose $\Lambda$ consists of $\Lambda_1,\cdots,\Lambda_m$. Since $\varphi$ is regular, all the $\Lambda_j$'s are isomorphic.
From the restriction of $\varphi$ to $\Lambda_1$, we can find a lift $\widetilde{P}$ of $P_{\bar{B}}^v$ in $\Aut(A_{\Lambda_1})$. Note that $\widetilde{P}$ can be considered as an element in $\Aut(\Lambda)$ by fixing all the vertices in $\Lambda_j$'s for $j\neq 1$. Let $\sigma_j$ be a deck transformation sending $\Lambda_1$ to $\Lambda_j$. Then $$\prod_{j=1}^{m}\sigma_j\widetilde{P}\sigma_j^{-1}$$ is a desired lift of $P_{\bar{B}}^v$.
\end{proof}


For $C = \bigsqcup_{i=1}^n B_i$ with $B_1, \dots, B_n \in C(v)$, we can naturally define $\bar{C}$ as $\bigcup_i \bar{B_i}$. 

\begin{corollary}\label{prop:sufficient_condition_partial}
Let $\varphi:\Lambda\to\Gamma$ be a regular covering map of graphs.
For a vertex $v \in \Gamma$, let $C=\bigsqcup_{i=1}^n B_i$ be a union of components of $\Gamma \setminus \st(v)$. Then, the partial conjugation $P_C^v$ is liftable if $\bar{C} = C$.
\end{corollary}
\begin{proof}
Since $\bar{C} = C$, we have 
$$\bigsqcup_{i=1}^n B_i = \bigcup_{i=1}^n \bar{B_i} = \bigsqcup_{i=1}^k \bar{B_i},$$ 
where in the last equation, $k \leq n$ is chosen in a minimal way, after possibly reordering the $B_i$'s.  
Then $P_C^v$ is the product of $P_{\bar{B_1}}^v, \cdots, P_{\bar{B_k}}^v$ where each $P_{\bar{B_i}}^v$ is liftable by Proposition~\ref{prop:partial_lift}. Therefore, $P_C^v$ is liftable.
\end{proof}

In Section \ref{sec:liftable_conjugating}, we show that the converse of Proposition~\ref{prop:sufficient_condition_partial} also holds (Proposition \ref{lift_partial_2}) if $\Gamma$ and $\Lambda$ have no isolated vertices.

\subsection{Examples}

\begin{figure}[ht]
\begin{center}
\begin{tikzpicture} 
\begin{scope}[shift={(0,-3)}]
\filldraw[black] (0,0) circle (3pt) node[anchor=west]{$v$}; 
\filldraw[black] (1,1.72) circle (3pt) node[anchor=east]{$v_1$}; 
\filldraw[black] (1,-1.72) circle (3pt) ; 
\filldraw[black] (-2,0) circle (3pt) node[anchor=east]{$v_3$}; 
\filldraw[black] (-3,1) circle (3pt) node[anchor=east]{$v_4$}; 
\filldraw[black] (-3,-1) circle (3pt) ; 
\filldraw[blue] (0.66,3.08) circle (3pt) node[anchor=east,black]{$v_2$}; 
\filldraw[blue] (2.36,2.08) circle (3pt) ; 
\filldraw[red] (0.66,-3.08) circle (3pt) ; 
\filldraw[red] (2.36,-2.08) circle (3pt) ; 
\draw (0,0)--(1,1.72)--(0.66,3.08)--(2.36,2.08) -- (1,1.72);
\draw (0,0)--(1,-1.72)--(0.66,-3.08)--(2.36,-2.08) -- (1,-1.72);
\draw (0,0)--(-2,0)--(-3,1)--(-3,-1) -- (-2,0);
\end{scope}

\filldraw[black] (-3,6) circle (3pt) node[anchor=west]{$u$}; 
\filldraw[black] (-5,6) circle (3pt) node[anchor=east]{$u_3$};
\filldraw[black] (-6,7) circle (3pt) node[anchor=east]{$u_4$}; 
\filldraw[black] (-6,5) circle (3pt) ; 

\filldraw[black] (-1.5,7.5) circle (3pt) node[anchor=west]{$u_1$}; 
\filldraw[black] (-1.5,4.5) circle (3pt) ; 

\filldraw[blue] (-0.5,8.5) circle (3pt) node[anchor=east,black]{$u_2$}; 
\filldraw[blue] (-0.5,6.5) circle (3pt) ; 

\filldraw[red] (-0.5,5.5) circle (3pt) ;
\filldraw[red] (-0.5,3.5) circle (3pt) ;

\filldraw[blue] (0.5,8.5) circle (3pt) ; 
\filldraw[blue] (0.5,6.5) circle (3pt) node[anchor=west,black]{$u'_2$};

\filldraw[red] (0.5,5.5) circle (3pt) ; 
\filldraw[red] (0.5,3.5) circle (3pt) ;

\filldraw[black] (1.5,7.5) circle (3pt) node[anchor=west]{$u'_1$}; 
\filldraw[black] (1.5,4.5) circle (3pt) ; 

\filldraw[black] (3,6) circle (3pt) node[anchor=east]{$u'$}; 
\filldraw[black] (5,6) circle (3pt) node[anchor=east]{$u'_3$};;
\filldraw[black] (6,7) circle (3pt) node[anchor=east]{$u'_4$};; 
\filldraw[black] (6,5) circle (3pt) ; 

\draw (-3,6)--(-5,6)--(-6,7)--(-6,5) -- (-5,6);
\draw (3,6)--(5,6)--(6,7)--(6,5) -- (5,6);

\draw (-3,6)--(-1.5,7.5)--(-0.5,8.5)--(0.5,8.5) -- (1.5,7.5) -- (3,6);
\draw (-3,6)--(-1.5,4.5)--(-0.5,3.5)--(0.5,3.5) -- (1.5,4.5) -- (3,6);
\draw (-1.5,7.5)--(-0.5,6.5)--(0.5,6.5)-- (1.5,7.5);
\draw (-1.5,4.5)--(-0.5,5.5)--(0.5,5.5)-- (1.5,4.5);

\draw[-{Stealth[scale=1.5]}] (0,2.5) -- (0,0.6) node[midway, right] {$\varphi$};
\end{tikzpicture}
\caption{The partial conjugation $P_{C}^v$, where $C$ is the subgraph induced by the two blue vertices, is not liftable. }
\label{figure:NotLift}
\end{center}
\end{figure}
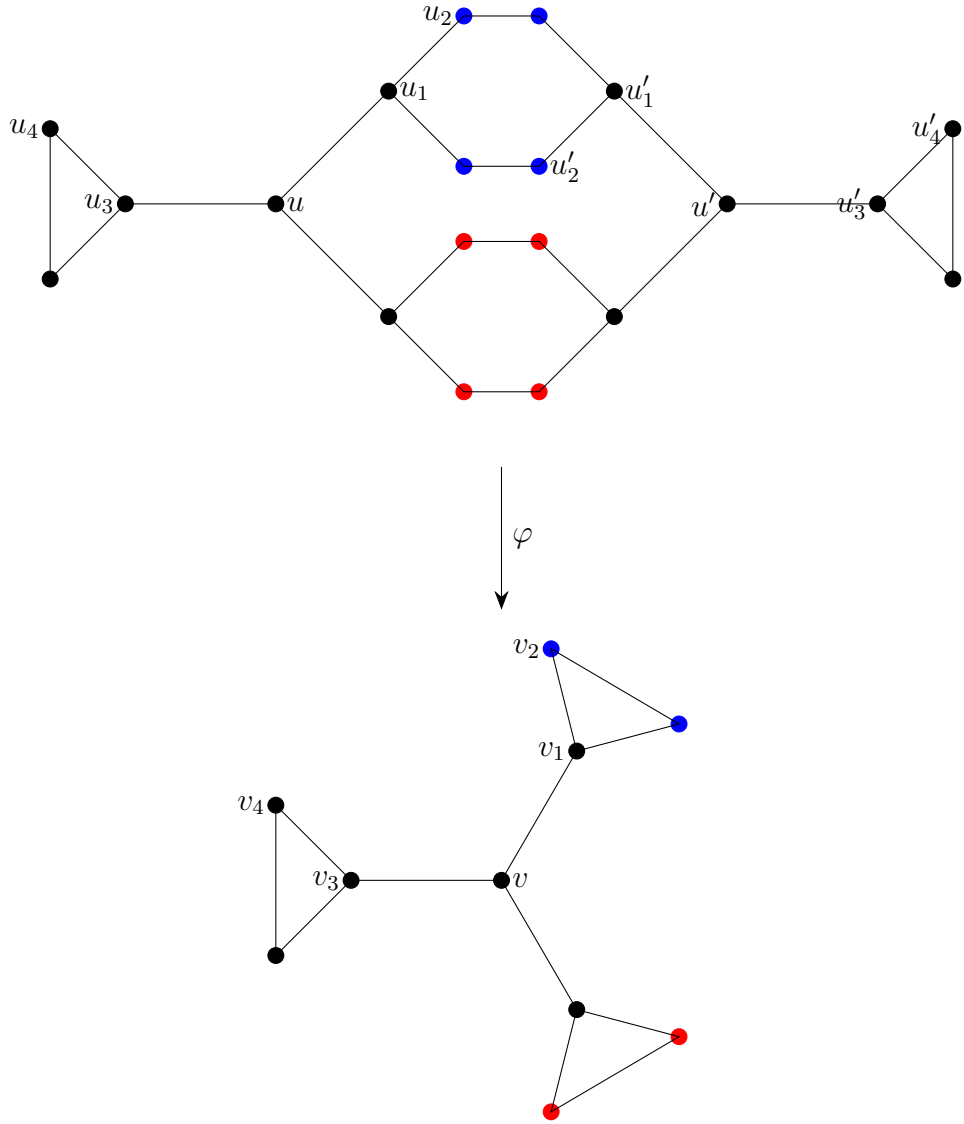
Before we conclude this section, we give an explicit example of transvections and partial conjugations that do not have lifts. Consider the regular covering map $\varphi:\Lambda\rightarrow\Gamma$ of graphs in Figure \ref{figure:NotLift}. 

\noindent\textbf{Transvections.} Consider two vertices $u_3$ and $u'_3$ which are in $\varphi^{-1}(v_3)$, and two vertices $u_4$ and $u'_4$ which are in $\varphi^{-1}(v_4)$.
Since $u_4\lesssim u_3$ and $u'_4\lesssim u'_3$, we have $v_4\lesssim_\varphi v_3$.
By Lemma \ref{lem:lift_transvection}, the transvection $T_{v_4}^{v_3}$ is liftable.
For two vertices $v_1,v_2\in\Gamma$, however, the definition of $\lesssim_\varphi$ says that $v_1$ and $v_2$ are not comparable under $\lesssim_\varphi$ even though $v_2\lesssim v_1$. It turns out that the transvection $T_{v_2}^{v_1}$ is not liftable by Corollary \ref{lem:lift_transvection2}.

\noindent\textbf{Partial conjugations.} Let $B_1$ and $B_2$ be the subgraphs of $\Gamma$ induced by the two blue vertices and the two red vertices, respectively. Let $C = B_1 \sqcup B_2$
Then we have $$\bar{B_1} = \bar{B_2} = \bar{C}= B_1 \sqcup B_2.$$
By Corollary \ref{prop:sufficient_condition_partial}, the partial conjugation $P_{C}^v=P_{B_1}^v P_{B_2}^v$ is liftable. However, neither $P_{B_1}^v$ nor $P_{B_2}^v$ are liftable, as proved later by Proposition \ref{lift_partial_2}.

%
%
%
%

\section{Liftable graph symmetries} \label{sec:lift_isom}
The goal of this section is to find a criterion for graph symmetries of $\Gamma$ to be liftable. We start with an easy example of liftable graph symmetries; permutations on a $\varphi$-equivalence classes.

\begin{lemma} \label{lem:permutataion_liftable}
If $[v]_{\varphi}$ is the $\varphi$-equivalence class for a vertex $v\in\Gamma$ and $\gamma_v$ is a permutation of $[v]_{\varphi}$, then the map $$\gamma(w) = \begin{cases} \gamma_v(w) & \text{if }w\in [v]_{\varphi}, \\ w, & \text{if }w \in V\Gamma \setminus [v]_{\varphi} \end{cases}$$ is a liftable graph symmetry of $\Gamma$.
\end{lemma}

\begin{proof}
By Lemma \ref{EquivalentClass}, the subgraph induced by $[v]_{\varphi}$ is complete or edgeless, and thus, $\gamma_v$ is a graph symmetry of the subgraph.
If a vertex $w$ in $V\Gamma \setminus [v]_{\varphi}$ is adjacent to a vertex in $[v]_{\varphi}$, then all the vertices of $[v]_{\varphi}$ are adjacent to $w$.
That is, $\gamma$ preserves the adjacency.
So $\gamma$ is a graph symmetry of $\Gamma$.

Let $[v]_{\varphi}=\{v_1,\cdots,v_n\}$ and let $u_1$ be a vertex in $\Lambda$ with $\varphi(u_1)=v_1$.
By the definition of $\varphi$-equivalence, for each $i>1$, there exists a unique vertex $u_i$ of $\Lambda$ such that $\varphi(u_i)=v_i$ and $u_i$ is equivalent to $u_1$.
It means that $\varphi^{-1}([v]_{\varphi})$ is a disjoint union of $V_1, \dots, V_k$ such that all the vertices of $V_j$ are equivalent and the restriction of $\varphi$ to $V_j$ is bijective onto $[v]_{\varphi}$ for $j=1,\cdots,k$.
Thus we can induce a graph symmetry $\tilde{\gamma}_j$ of $V_j$ so that $\varphi\circ\tilde{\gamma}_j(u) = \gamma\circ\varphi(u)$ for all $u \in V_j$.
Let 
$$\tilde{\gamma}(u) = \begin{cases} \tilde{\gamma}_i(u) & \text{if }u \in V_j, \\ u, & \text{otherwise}.\end{cases}$$
Then we have $\varphi\circ \tilde{\gamma}(u) = \varphi\circ \gamma(u)$ and therefore, $\tilde{\gamma}$ is a lift of $\gamma$.
\end{proof}

For a word $w = w_1^{n_1} \dots w_k^{n_k}$ in $A_\Lambda$ with $w_1,\cdots,w_k\in V\Lambda$, the cardinality of the support of $\phi(w)$ is $\leq k$ since every cancellation on $w$ may occur on $\phi(w)= \phi(w_1^{n_1} \dots w_k^{n_k}) = \varphi(w_1)^{n_1} \dots \varphi(w_k)^{n_k}$.
That is, for every word $w \in A_\Lambda$, we have $\supp \phi(w) \subseteq \varphi(\supp w)$. Indeed, we can say more:

\begin{lemma} \label{lem:cover_essupp}
For any reduced word $w \in A_\Lambda$, we have $\esupp \phi(w) \subseteq \varphi ( \esupp w )$.
\end{lemma}

\begin{proof}
Let $x \in A_\Lambda$ be a cyclically reduced word conjugate to $w$, and let $y \in A_\Gamma$ be a cyclically reduced word conjugate to $\phi(x)$.
Since $\phi(w)$ is conjugate to $y$, we have $\esupp \phi(w) = \supp y \subseteq \supp \phi(x) \subseteq \varphi (\supp x) = \varphi(\esupp w)$.
\end{proof}

The following lemma is the analog of Proposition \ref{prop:essupp_auto} for a regular covering map of graphs.

\begin{lemma} \label{lem:lift_essupp}
For a liftable automorphism $f$ of $A_\Gamma$, let $\sigma$ be a graph symmetry of $\Gamma$ such that $f\sigma$ is essential. For each vertex $v\in\Gamma$, then we have $\esupp f(v) \subseteq V\Gamma_{\gtrsim_\varphi \sigma^{-1}(v)}$.
\end{lemma}

\begin{proof}
Let $\tilde{f}$ be a lift of $f$.
For each $u \in \varphi^{-1}(v)$, we have $\phi\circ \tilde{f}(u)=f\circ\phi (u)= f(v)$ so that by Lemma \ref{lem:cover_essupp}, the $\varphi$-image of $\esupp \tilde{f}(u)$ contains $\esupp f(v)$.
Since $\sigma^{-1}(v)$ is contained in $\esupp f(v)$, there is a vertex $u' \in \varphi^{-1}(\sigma^{-1}(v))$ which is contained in $\esupp \tilde{f}(u)$.
Note that $\rank(\tilde{f}(u))=|\st(u)|=|\st(v)|=|\st(\sigma^{-1}(v))|=|\st(u')|$.
By Proposition \ref{prop:essupp_auto}, then, we have $\esupp \tilde{f}(u) \subseteq V\Lambda_{\gtrsim u'}$. 
So by Lemma \ref{lem:cover-preserving}, we obtain the following inclusion:
$$\esupp f(v) \subseteq \varphi(\esupp \tilde{f}(u)) \subseteq \varphi(V\Lambda_{\gtrsim u'}) = V\Gamma_{\gtrsim_\varphi \sigma^{-1}(v)}.$$
Therefore, the lemma holds.
\end{proof}

Recall that an automorphism $F$ of $A_\Lambda$ is said to be \textit{fiber-preserving} if it is a lift of an automorphism $f$ of $A_\Gamma$, and the set of all fiber-preserving automorphisms of $A_\Lambda$ forms a subgroup of $\Aut(A_\Lambda)$, denoted by $\FAut(\varphi)$. 
By applying Theorem~\ref{thm:iso_thm} to both $F$ and $f$, we can find a sufficient condition for a graph symmetry to be liftable.

\begin{theorem}\label{thm:lift_iso}
For a regular covering map $\varphi: \Lambda \to \Gamma$ of graphs, let $f\in \LAut(\varphi)$ be a liftable automorphism with its lift $F\in\FAut(\varphi)$.
Then any graph symmetry $\sigma$ which makes $f\sigma^{-1}$ essential admits a lift $\mu\in\FAut(\varphi)$ which is a graph symmetry of $\Lambda$ such that $\varphi\circ\mu=\sigma\circ\varphi$. 
Moreover, if $\Gamma$ has no isolated vertex, then such $\mu$ makes $F \mu^{-1}$ essential.
\end{theorem}
\begin{proof}
By Theorem \ref{thm:iso_thm}, there is a graph symmetry $\tau$ of $\Lambda$ such that $F\tau^{-1}$ is essential.
For any vertex $u\in\Lambda$, since $|\st(\tau(u))|=|\st(u)|$, by Proposition \ref{prop:essupp_auto}, we have $$\tau(u) \in \esupp F(u) \subseteq V\Lambda_{\gtrsim \tau(u)}.$$
By Lemma \ref{lem:lift_essupp}, we have $\sigma(\varphi(u))\in\esupp f(\varphi(u)) \subseteq V\Gamma_{\gtrsim_\varphi \sigma(\varphi(u))}$.
Since Lemma \ref{lem:cover_essupp} implies $\esupp f(\varphi(u)) \subseteq \varphi ( \esupp F(u))$, $\esupp F(u)$ must contain at least one vertex $u'$ of $\varphi^{-1}(\sigma(\varphi(u)))$ such that $\tau(u) \lesssim u'$.
But the fact that $$\deg \tau(u) = \deg u =\deg \varphi(u) =\deg \sigma(\varphi(u)) = \deg u'$$
implies that $u'$ is equivalent to $\tau(u)$.

Let $\Lambda^i$ be the set of isolated vertices of $\Lambda$ and $\Lambda^c$ the complement of $\Lambda^i$ in $\Lambda$.
Note that any graph symmetry preserves the set of isolated vertices and $\varphi$ maps the set of isolated vertices in $\Lambda$ to the set of isolated vertices in $\Gamma$.
It means that the restriction $\sigma_i$ of $\sigma$ to $\varphi(\Lambda^i)$ always has a graph symmetry $\mu_i$ of $\Lambda^i$ such that $\varphi\circ\mu_i=\sigma_i\circ\varphi$.
Hence, we will show that the restriction $\sigma_c$ of $\sigma$ to $\varphi(\Lambda^c)$ has a graph symmetry $\mu_c$ of $\Lambda^c$ such that  $\varphi\circ\mu_c=\sigma_c\circ\varphi$. 

By Lemma \ref{lem:simplicial_criterion}, if $u\in\Lambda^c$, then the vertex $u'$ obtained from $u$ in the first paragraph is unique since the set
\begin{equation}\label{eqn:unique}
    [\tau(u)] \cap \varphi^{-1}(\sigma(\varphi(u))) \tag{$*$}
\end{equation}
contains at most one vertex, and thus we have a map $\mu_c:V\Lambda^c\rightarrow V\Lambda^c$ sending $u$ to $u'$ (in particular, $\esupp F(u)$ contains $\mu_c(u)$).
Let $u_1,u_2\in V\Lambda^c$ be two distinct vertices. 
If $\varphi(u_1)=\varphi(u_2)$, then $$d(u_1,u_2)=d(\tau(u_1),\tau(u_2))\geq 3$$ so that $[\tau(u_1)]\cap[\tau(u_2)]=\emptyset$.
Otherwise, $\varphi^{-1}(\sigma(\varphi(u_1)))\cap \varphi^{-1}(\sigma(\varphi(u_2)))=\emptyset$.
Thus, we have $\mu_c(u_1)\neq\mu_c(u_2)$, i.e.,  
$$\left([\tau(u_1)] \cap \varphi^{-1}(\sigma(\varphi(u_1)))\right)\cap \left([\tau(u_2)] \cap \varphi^{-1}(\sigma(\varphi(u_2)))\right)=\emptyset,$$
which implies that $\mu_c$ is bijective.
If $u_1$ and $u_2$ are joined by an edge, by Lemma~\ref{lemma:adjacent}, $\mu_c(u_1)$ and $\mu_c(u_2)$ must be adjacent since $\tau(u_1)$ is adjacent to $\tau(u_2)$ and we have $\mu_c(u_i)\in[\tau(u_i)]$ and $\mu_c(u_1)\neq\mu_c(u_2)$.
Hence, $\mu_c$ is the desired graph symmetries of $\Lambda^c$.

For the `moreover' part, the assumption that $\Gamma$ has no isolated vertex implies that $\Lambda^c$ is equal to $\Lambda$ and the graph symmetry $\mu_c$ of $\Lambda^c$ obtained in the previous paragraph becomes the desired graph symmetry $\mu$ of $\Lambda$.
\end{proof}

Indeed, the above theorem implies that a graph symmetry, considered as an element of $\Aut(A_\Gamma)$, is liftable if and only if it can be lifted through the covering map $\varphi$.

\begin{corollary}\label{cor:lift_iso}
Let $\varphi: \Lambda \to \Gamma$ be a regular covering map of graphs and let $\sigma$ be a graph symmetry of $\Gamma$. 
Then $\sigma$ is in $\LAut(\varphi)$ if and only if there is a graph symmetry $\mu$ of $\Lambda$ such that $\varphi\circ\mu=\sigma\circ\varphi$.
\end{corollary}
\begin{proof}
The `if' direction is obvious and the other direction can be proven by replacing $f$ in Theorem~\ref{thm:lift_iso} by $\sigma$.
\end{proof}

\begin{corollary} \label{lem:deck_lift}
Let $\varphi:\Lambda\rightarrow\Gamma$ be a regular covering map of graphs without isolated vertices.
Let $f\in\LAut(\varphi)$ be an essential automorphism. For any lift $F$ of $f$, then, there exists a deck-transformation $\mu$ such that $F\mu$ is an essential lift of $f$.
\end{corollary}
\begin{proof}
Since $f$ is essential, we can choose $\sigma$ as the identity graph symmetry in Theorem~\ref{thm:lift_iso}. Then the graph symmetry $\mu$ of $\Lambda$ constructed from $\sigma$ and $F$ is a deck transformation since $\varphi\circ\mu=\sigma\circ\varphi$.
\end{proof}

In the proof of Theorem~\ref{thm:lift_iso}, if $\Gamma^i$ has only one isolated vertex, then the restriction $\mu_i$ of $\mu$ to $\Lambda^i$ can be chosen as the restriction of $\tau$ to $\Lambda^i$ so that the last statement of the theorem holds.
However, if $\Gamma^i$ has at least two isolated vertices, then the statement may not hold since the set (\ref{eqn:unique}) in the proof is no longer unique.

\begin{example}\label{Ex:FreeGroupGraphSymmetry}
Let $\Gamma=\{a,b\}$ be a graph consisting of two isolated vertices.
There are two graph symmetries $\sigma_1$ and $\sigma_2$ of $\Gamma$ where $\sigma_1$ is the identity and $\sigma_2$ is to exchange two vertices.
Let $f$ be an automorphism of $A_\Gamma$ sending $a$ and $b$ to $ab$ and $a^2b$, respectively. Then both $\sigma_1$ and $\sigma_2$ satisfy Theorem~\ref{thm:iso_thm}.

Let $\Lambda=\{a_1,a_2,b_1,b_2\}$ be a regular covering of $\Gamma$ with covering map sending $a_i$ to $a$ and $b_i$ to $b$ for $i=1,2$.
Let $F$ be an automorphism of $\Lambda$ sending $a_1$, $a_2$, $b_1$ and $b_2$ to $a_1b_1$, $a_1b_2$, $a_2a_1b_1$ and $a_2^2b_1$ (it can easily be seen that $\{a_1b_1, a_1b_2, a_2a_1b_1, a_2^2b_1\}$ is a basis of $A_\Lambda$ so that $F$ is an automorphism).
Even though there exists a lift of $\sigma_2$ such that Theorem~\ref{thm:iso_thm} holds, any lift of $\sigma_1$ does not satisfy Theorem~\ref{thm:iso_thm}.
\end{example}

%
%
%
%

\section{Liftable conjugating automorphisms}\label{sec:liftable_conjugating}

Recall that a conjugating automorphism is an automorphism sending each generator to its conjugate and that the group consisting of all conjugating automorphisms of a RAAG is in fact the subgroup generated by partial conjugations.
In this section, we show that for a regular covering map $\varphi \colon \Lambda \to \Gamma$ of graphs, the group of liftable conjugating automorphisms of $A_\Gamma$, denoted by $\LPAut(\varphi)$, is generated by liftable partial conjugations. Since this fact is obtained based on the results in Section \ref{sec:lift_isom}, we additionally assume that $\Gamma$ and thus $\Lambda$ have no isolated vertices.

We first revisit the way how Laurence in \cite{MR1356145} proved Theorem \ref{thm:conjugating}.
Let $f\in\Aut(A_\Gamma)$ be a non-trivial conjugating automorphism.
For each vertex $v_i\in V\Gamma$, $f(v_i)$ can be represented by a reduced word $w_i v_i w_i^{-1}$; in this case, $w_i$ is said to be a \textit{conjugating word} for $f(v_i)$. 
Then $|f|$ is defined as the sum of the length of $w_i$ for all $v_i\in V\Gamma$. 
In order to prove the theorem, we need the following facts.

\begin{lemma}{\cite[Lemmas 2.5 and 2.8]{MR1356145}}\label{LL}
Let $C$ be a component of $\Gamma\setminus\st(v_1)$. For a vertex $v_2\in C$, if $w_1 v_1^{\varepsilon}$ is a left-most vertex in a conjugating word for $f(v_2)$ and $v_3$ is a vertex in $C$, then $w_1 v_1^{\varepsilon}$ is also a left-most vertex in a conjugating word for $f(v_3)$.

Moreover, such a pair $v_1$ and $v_2$ always exists.
\end{lemma}

Based on Lemma \ref{LL}, the Laurence's algorithm starts from $f_1:=f$ as follows:
\begin{enumerate}
    \item\label{LL1} From a conjugating automorphism $f_i$, we find two vertices $v_{i_1}$ and $v_{i_2}$ such that when $f(v_{i_1})$ is represented by a reduced word $w_{i_1} v_{i_1} w_{i_1}^{-1}$, $w_{i_1}v_{i_1}^{\varepsilon_i}$ is a left-most word in a conjugating word for $f(v_{i_2})$ where $\varepsilon_i\in\{-1,1\}$.
    \item\label{LL2} Choose the component $C_{i_1}$ of $\Gamma\setminus\st(v_{i_1})$ containing $v_{i_2}$, and let $P_i=P_{C_{i_1}}^{v_{i_1}}$ and $f_{i+1}=f_iP_i^{\varepsilon_i}$. Then Lemma \ref{LL} implies that $|f_{i+1}|<|f_i|$.
    \item We go to Item \noindent\eqref{LL1} and replace $f_i$ by $f_{i+1}$ and do the same process until the automorphism becomes the identity.
\end{enumerate}
The algorithm terminates when $f P_1^{\varepsilon_1}\cdots P_n^{\varepsilon_n}$ becomes the identity and $f$ is represented by $P_{1}^{-\varepsilon_1}\cdots P_{n}^{-\varepsilon_n}$. This is the proof of Theorem \ref{thm:conjugating}.


Since conjugating automorphisms are essential, by Corollary~\ref{lem:deck_lift} every liftable conjugating automorphism $f\in \LAut(\varphi)$ admits an essential lift $F\in \FAut(\varphi)$. We prove that in fact, it even admits a conjugating lift.

\begin{lemma} \label{lem:conj_lift}
Every liftable conjugating automorphism admits a conjugating lift.
\end{lemma}
\begin{proof}
Let $f \in \LPAut(\varphi)$ be given and $F$ an essential lift of $f$ obtained by Corollary~\ref{lem:deck_lift}. Suppose $u\in V\Lambda$ is maximal. Note that $\esupp(\tilde{f}(u))$ is contained in $V\Lambda_{\gtrsim u}$, which is equal to the equivalence class of $u$, and the restriction of $\varphi$ to this equivalence class induces a bijection. 
It means that $\varphi(\esupp(\tilde{f}(u)))=\esupp(f(\varphi(u)))=\{\varphi(u)\}$ and thus $\esupp(\tilde{f}(u))=\{u\}$. Thus, $\tilde{f}$ sends maximal vertices to conjugates of themselves.

Let $\Lambda_1$ be the subgraph of $\Lambda$ obtained by removing maximal vertices, and let $\Gamma_1=\varphi(\Lambda_1)$. Then the restriction of $\varphi$ to $\Lambda_1$ is a regular covering map $\Lambda_1\to\Gamma_1$.
Note that the automorphism $f$ of $A_\Gamma$ induces an automorphism $f_1$ of $A_{\Gamma_1}$ by the composition of the restriction of $f$ to $A_{\Gamma_1}$ and the quotient map $A_\Gamma\rightarrow A_{\Gamma_1}$. Similarly, $\tilde{f}$ induces an automorphism $\tilde{f}_1$ of $A_{\Lambda_1}$ such that $\tilde{f}_1$ is a lift of $f_1$ via $\varphi_1$. As in the previous paragraph, we can deduce that $\tilde{f}_1$ sends its maximal vertices to conjugates of themselves. By repeating this process, we can obtain a conjugating lift of $f$.
\end{proof}

Before we move on, we note the following lemmas.

\begin{lemma}\label{conjugating}
Let $v_1,v_2,v_3$ be three consecutive vertices in a path in $\Gamma$ such that $d(v_1,v_3)=2$. Let $f$ be a conjugating automorphism fixing $v_1$. Then $f(v_2)$ is represented by a reduced word $wv_2w^{-1}$ where $w\in\langle\lk(v_1)\rangle$. Moreover, $f(v_3)$ is represented by a reduced word $xyv_3y^{-1}x^{-1}$ where $x\in\langle\lk(v_1)\rangle$ and $y\in\langle\st(v_2)\rangle$.
\end{lemma}
\begin{proof}
The first part of the statement is basically the content of \cite[Lemma 2.9]{MR1356145}. Suppose $f(v_3)$ is represented by a reduced word $zv_3z^{-1}$ and $w'$ is a longest left-most word of $z$ which is also a left-most word of $w$; note that $w'\in\langle\lk(v_1)\rangle$.
From the fact that $f(v_2)$ and $f(v_3)$ commute, we will show the second statement of the lemma.
If $w'=w=z$, then the statement holds. 
Otherwise, there are three cases: 
\begin{enumerate}
    \item $w'=w$ and $z$ is a reduced word of the form $w'z'$ for a non-trivial reduced word $z'$. In this case, $z'$ must commute with $v_2$ so that $z'\in\langle\st(v_2)\rangle$.
    \item $w'=z$ and $w$ is a reduced word of the form $w'w''$ for a non-trivial reduced word $w''$. In this case, $z'$ must be trivial and $w''$ must commute with $v_3$.
    \item $w$ and $z$ are reduced words of the form $w'w''$ and $w'z'$ for non-trivial reduced words $w''$ and $z'$, respectively. In this case, $z'$ must commute with both $w''$ and $v_2$. It means that $z'$ is in $\langle\st(v_2)\rangle$.
\end{enumerate}
In every case, we see that $z$ is a reduced word of the form $xy$ for $x\in\langle\lk(v_1)\rangle$ and $y\in\langle\st(v_2)\rangle$.
\end{proof}

\begin{theorem} \label{thm:liftable_conjugating}
For a regular covering map $\varphi: \Lambda \to \Gamma$ of graphs without isolated vertices, liftable partial conjugations generate $\LPAut( \varphi)$.
\end{theorem}

\begin{proof}
Let $f$ be a liftable partial conjugation in $\LAut(\Gamma)$. 
Since $f$ is a conjugating automorphism, by using Laurence's algorithm, we can obtain a sequence of partial conjugations $(P_{C_1}^{v_1})^{\varepsilon_1},\cdots,(P_{C_n}^{v_n})^{\varepsilon_n}$ such that $f(P_{C_1}^{v_1})^{\varepsilon_1},\cdots,(P_{C_n}^{v_n})^{\varepsilon_n}$ is the identity. Additionally, we assume that each $(P_{C_i}^{v_i})^{\varepsilon_i}$ is maximal in the following sense: in the algorithm, each $C_i$ can be allowed to be a union of components (rather than a component) so that $C_i$ is chosen to be maximal under the set inclusion in Item (\ref{LL2}) of the algorithm; for instance, $C_1$ is the set of vertices $v$ in $\Gamma\setminus\st(v_1)$ such that $w_1v_1^{\varepsilon_1}$ is a left-most word in a conjugating word for $f(v)$.

Assume that $(P_{C_1}^{v_1})^{\varepsilon_1}$ is not liftable. 
Suppose $C_1=B_1\sqcup\cdots\sqcup B_m$ where $B_i$ is a component of $\Gamma\setminus\st(v_1)$. 
Corollary \ref{prop:sufficient_condition_partial} implies that after possibly reordering the $B_i$'s, $\bar{B_1}$ contains $B_0\in\mathcal{C}(v_1)$ which is not any of $B_1,\cdots,B_m$.  
Since $\bar{B_0}=\bar{B_1}$, by Lemma~\ref{lem:Minimality} and the paragraph below that lemma, we can choose a component $\widetilde{B}_0$ of $\varphi^{-1}(B_0)$ and a component $\widetilde{B}_1$ of $\varphi^{-1}(B_1)$ such that they are adjacent to $\st(u)$ for some $u\in\varphi^{-1}(v_1)$ and contained in the same component of $\Lambda\setminus\st(u)$.

By Lemma \ref{lem:conj_lift}, there is a conjugating lift $F$ of $f$. Let $\tilde{\iota}$ be an inner automorphism in $\Aut(A_\Lambda)$ such that $\tilde{\iota}F$ fixes the vertex $u$. By Lemma \ref{lem:inner}, $\tilde{\iota}$ induces an inner automorphism $\iota$ on $A_\Gamma$ such that $\tilde{\iota}$ is a lift of $\iota$ and, in particular, $\tilde{\iota}F$ is a lift of $\iota f$.
We note that when we apply Laurence's algorithm to $\iota f$, $P_{C_1}^{v_1}$ can be a first partial conjugation obtained in the process, as it is a first partial conjugation obtained from $f$, i.e. $|\iota f|>|\iota f P_{C_1}^{v_1}|$.

Let $b_0\in \widetilde{B}_0$ and $b_1\in \widetilde{B}_1$ be vertices which are at distance 2 from $u$, and let $u_i$ be a vertex in $\lk(u)$ which is adjacent to $b_i$ for $i=0,1$ (possibly, $u_0=u_1$). Let $z_i b_i z_i^{-1}$ be a reduced word representing $\tilde{\iota}F(b_i)$. By Lemma \ref{conjugating}, then $z_i$ is represented by $x_iy_i$ where $x_i\in\langle\lk(u)\rangle$ and $y_i\in\langle\st(u_i)\rangle$.
Let $\iota f(\varphi(b_i))$ be represented by a reduced word $w_i\varphi(b_i)w_i^{-1}$.
Note that as group elements, $$w_i\varphi(b_i)w_i^{-1}=\phi(x_i)\phi(y_i)\varphi(b_i)\phi(y_i^{-1})\phi(x_i^{-1})$$ where $\phi(x_i)\in\langle\lk(\varphi(u))\rangle$ and $\phi(y_i)\in\langle\st(\varphi(u_i))\rangle$.
Since $v_1$ is a left-most word in $w_1$ by the algorithm, $\esupp(\phi(y_1))$ contains $v_1$ and thus $\esupp(y_1)$ must contain $u$.
Since $u$ commutes with $x_1$, $u$ is a left-most word in $z_1$. 
Since $\widetilde{B}_0$ and $\widetilde{B}_1$ are contained in the same component of $\Lambda\setminus\st(u)$, by Lemma~\ref{LL}, $u$ is also a left-most word in $z_0$.
Since $\varphi(u)$ and $\varphi(b_0)$ do not commute in $A_\Gamma$, $v_1$ must be a left-most word in $w_0$. 
However, this is a contradiction with the maximality of $C_1$.
Thus $(P_{C_1}^{v_1})^{\varepsilon_1}$ must be liftable.

Since the partial conjugation $(P_{C_2}^{v_2})^{\varepsilon_2}$ is obtained from $f(P_{C_1}^{v_1})^{\varepsilon_1}$ by applying Laurence's algorithm, by the same reason as for $(P_{C_1}^{v_1})^{\varepsilon_1}$, we can know that $(P_{C_2}^{v_2})^{\varepsilon_2}$ is also liftable. By induction on $n$, all $(P_{C_i}^{v_1})^{\varepsilon_i}$'s are liftable and therefore, $f$ must also be liftable.
\end{proof}

The proof of Theorem \ref{thm:liftable_conjugating} implies that the converse of Corollary \ref{prop:sufficient_condition_partial} also holds if $\Gamma$ has no isolated vertex. Thus, we obtain the following proposition.

\begin{proposition}\label{lift_partial_2}
Let $\varphi:\Lambda\rightarrow\Gamma$ be a regular covering map of graphs without isolated vertices.
For any vertex $v\in \Gamma$ and a union $C$ of components of $\Gamma\setminus\st(v)$, then, the partial conjugation $P_C^v$ is liftable if and only if $\bar{C} = C$.
\end{proposition}

\section{Finite generation of a liftable automorphism group} \label{sec:gen_liftable}
Until now, we have obtained sufficient conditions for elementary automorphisms of $A_\Gamma$ to be liftable. The following is a summary which is nothing but an amalgamation of Lemma \ref{lem:lift_sym_inv}, Lemma \ref{lem:lift_transvection}, Corollary \ref{prop:sufficient_condition_partial} and Corollary~\ref{cor:lift_iso}.

\begin{theorem}\label{prop:elementary_lift}
For a regular covering map $\varphi: \Lambda \to \Gamma$ of graphs, the following hold.
\begin{enumerate}
\item For every $v \in V\Gamma$, the inversion of $v$ is liftable.
\item A graph symmetry $\sigma$, considered as an element of $\Aut(A_\Gamma)$, is liftable if and only if it can be lifted through the covering map $\varphi$.
\item For two vertices $v, v' \in V\Gamma$, if $v \gtrsim_\varphi v'$, the transvection $T_{v'}^v$ is liftable.
\item For $v \in V\Gamma$ and $C$ an union of components of $\Gamma \setminus \st(v)$, the partial conjugation $P_C^v$ is liftable if $\bar{C} = C$. 
\end{enumerate}
\end{theorem}

If we additionally assume that $\Gamma$ has no isolated vertices, then we can show the finite generation of $\LAut( \varphi)$ by following the philosophy of Laurence's proof of finite generation of the automorphism group of a RAAG.

\begin{theorem} \label{thm:lifable_group}
Let $\varphi: \Lambda \to \Gamma$ be a regular covering map of graphs without isolated vertices.
Then $\LAut( \varphi)$ is generated by liftable graph symmetries, inversions, liftable transvections, and liftable partial conjugations.
\end{theorem}

\begin{proof}
Choose $f \in \LAut( \varphi)$. By Theorem~\ref{thm:lift_iso}, we can assume that $f$ is essential, and by Lemma~\ref{OrderOnV}, suppose that the vertices $v_1, \cdots, v_k$ of $\Gamma$ are ordered such that if $v_i \gnsim v_{i'}$, then $i < {i'}$. Let $\mathcal{E}$ be the subgroup of $\Aut(A_\Gamma)$ generated by liftable graph symmetries, inversions, liftable transvections, and liftable partial conjugations.
We claim that for each $j\in\{1,\cdots,k\}$, there exist $g_{j}, h_{j} \in \mathcal{E}$ such that $f_j=g_j f h_j$ is essential and $f_{j}(v_i)$ is conjugate to $v_i$ for all $i \leq j$.
If this claim is true, then $f_k$ is a liftable conjugating automorphism. By Theorem \ref{thm:liftable_conjugating}, then we can conclude that $f$ is a product of liftable elementary automorphisms.

To show the claim, we use an induction on the indices of vertices.
By Theorem~\ref{thm:lift_iso}, there exists an essential lift $\tilde{f}$ of $f$. 
By Lemma~\ref{lem:cover-preserving}, any vertex $u_1\in\varphi^{-1}(v_1)$ is maximal in $V\Lambda$. 
Let $\tilde\iota_1$ be an inner automorphism satisfying that $\tilde{f}\tilde\iota_1(u_1)$ is cyclically reduced.
By Proposition \ref{prop:essupp_auto}, then $\supp(\tilde{f}\tilde{\iota}_1(u_1))$ is a subset of $[u_1]$.
Let $\iota_1 = \varphi\circ \tilde\iota_1 \circ\varphi^{-1}$ be an inner automorphism of $A_\Gamma$, which is well-defined by Lemma \ref{lem:inner}.
Then we have $\supp(f\iota_1(v_1))=\varphi(\supp(\tilde{f}\tilde{\iota}_1(u_1)))\subseteq [v_1]_{\varphi}$.
Hence there exists a product $t_1$ of liftable transvections (or their inverses) such that $t_1f\iota_1(v_1) \in \{ v_1, v_1^{-1} \}$.
For the inversion $c_1$ of $v_1$, then $c_1^{\ell_1}t_1f\iota_1(v_1)$ is equal to $v_1$ for some $\ell_1\in\{0,1\}$.
Since $f_1=c_1^{\ell_1}t_1f\iota_1$ is essential, $g_1 := c_1^{\ell_1} t_1$ and $h_1 := \iota_1$ are what we want.

The proof of the induction step is similar to the above.
Assume for some $j \in \{ 2, \dots, k \}$, there exist $g_{j-1}, h_{j-1} \in \mathcal{E}$ such that the claim holds.
By Theorem \ref{thm:lift_iso}, then there exists an essential lift $\tilde{f}_{j-1}$ of $f_{j-1}$.
Let $\tilde\iota_{j}$ be an inner automorphism of $A_\Gamma$ so that $\tilde{f}_{j-1} \tilde\iota_j(u_j)$ is cyclically reduced for some $u_j \in \varphi^{-1}(v_j)$.
By Proposition \ref{prop:essupp_auto}, we have $\supp \tilde{f}_{j-1} \tilde\iota_j(u_j)\subseteq V\Gamma_{\gtrsim u_j}$, which implies $\supp f_{j-1}\iota_j(v_j) \subseteq V\Lambda_{\gtrsim_{\varphi} v_j}$ for $\iota_j=\varphi\circ\tilde{\iota}_j\circ\varphi^{-1}$.
So there exists a product of liftable transvections, denoted by $t_j$, such that $t_jf_{j-1}\iota_j(v_j) \in \{ v_j, v_j^{-1} \}$.
Therefore, $c_jt_jf_{j-1}\iota_j(v_j)$ is equal to $v_j$ for some power of the inversion $c_j: v_j \mapsto v_j^{\pm 1}$.
Hence, the claim holds.
\end{proof}

%
%
%
%
%

\section{Lifts of the identity} \label{sec:lift_identity}

Associated to a regular covering map $\varphi \colon \Lambda \to \Gamma$ of graphs, by definition, 
we have the following short exact sequence
$$
1 \to \FD(\varphi) \to \FAut(\varphi) \to \LAut(\varphi) \to 1
$$
where $\FD(\varphi)$ is the group containing all the automorphisms $F \in \Aut(A_\Lambda)$ which are lifts of the identity in $\Aut(A_\Gamma)$. 
If we consider the deck transformation group $\Deck(\varphi)$ as a subgroup of $\Aut(A_\Lambda)$, it is obviously contained in $\FD(\varphi)$.
The goal of this section is to find what type of automorphisms are in $\FD(A_\Lambda)$ and study its structure.

The covering map $\varphi$ induces a surjective homomorphism $\phi \colon A_\Lambda \to A_\Gamma$. We thus also have a short exact sequence $1 \to \ker(\phi) \to A_\Lambda \to A_\Gamma \to 1$. It follows that for any element $k \in \ker(\phi)$, the inner automorphism of $A_\Lambda$ by $k$ is in $\FD(\varphi)$.

But we can find more surprising automorphisms in $\FD(\varphi)$. Let $x,y,z$ be vertices in $\Lambda$ such that $\lk(x) \subseteq \st(y) \cap \st(z)$. We denote by $T_x^{[y,z]}$ the automorphism that sends $x$ to $x[y,z]$ and fixes all other vertices. This kind of automorphisms are called \textit{commutator transvections}. Note that $T_x^{[y,z]}$ is non-trivial if and only if $y$ and $z$ do not commute. Hence, if $y$ and $z$ do not commute in $A_\Lambda$ but their images $\phi(y)$ and $\phi(z)$ commute in $A_\Gamma$, $T_x^{[y,z]}$ is in $\FD(\varphi)$. This can happen as shown in Figure \ref{CommutatorTransvection}. Indeed, the commutator transvections $T_{x_i}^{[y_i,z_i]}$ for $i = 1,2$ are lifts of $T_x^{[y,z]}$ but since $y$ and $z$ are adjacent, the latter is just the identity.
\begin{figure}
\begin{center}
\begin{tikzpicture} 

\filldraw[black] (-4,-0.5) circle (2pt) ; 
\filldraw[black] (-4,-2) circle (2pt) node[anchor=west]{$x_1$}; 
\filldraw[black] (-5,0.5) circle (2pt) node[anchor=west]{$y_1$}; 
\filldraw[black] (-3,0.5) circle (2pt) node[anchor=east]{$z_1$}; 
\filldraw[black] (-3,1.8) circle (2pt) node[anchor=east]{$y_2$}; 
\filldraw[black] (-5,1.8) circle (2pt) node[anchor=west]{$z_2$}; 
\filldraw[black] (-4,2.8) circle (2pt) ; 
\filldraw[black] (-4,4.3) circle (2pt) node[anchor=west]{$x_2$}; 

\filldraw[black] (-2,2.1) circle (2pt) ; 
\filldraw[black] (-2.8,2.8) circle (2pt) ; 

\filldraw[black] (-5.8,0.2) circle (2pt) ; 
\filldraw[black] (-5.2,-0.5) circle (2pt) ; 

\draw (-4,-2) -- (-4,-0.5) -- (-5,0.5)-- (-5, 1.8) --(-4,2.8) --(-3,1.8) -- (-3,0.5) -- (-4,-0.5);
\draw (-4,2.8)--(-4,4.3);
\draw (-3,1.8) -- (-2,2.1) -- (-2.8,2.8) -- (-3,1.8);
\draw (-5,0.5) -- (-5.8,0.2) -- (-5.2,-0.5) -- (-5,0.5);

\filldraw[black] (4,1) circle (2pt) ; 
\filldraw[black] (4,-1) circle (2pt)node[anchor=west]{$x$} ;
\filldraw[black] (3,2.5) circle (2pt) node[anchor=north]{$y$};
\filldraw[black] (5,2.5) circle (2pt) node[anchor=west]{$z$};
\filldraw[black] (2,2.8) circle (2pt);
\filldraw[black] (2.8,3.5) circle (2pt);
\draw (4,1) -- (4,-1);
\draw (4,1) -- (3,2.5) -- (5,2.5) -- (4,1);
\draw (3,2.5)--(2,2.8)--(2.8,3.5)--(3,2.5);

\draw[-stealth]node[above] {$\varphi$}(-1.5,1) -- (1.5,1) ;

\end{tikzpicture}
\caption{An example of a commutator transvection $T_{x_i,[y_i,z_i]}$ which is a lift of the identity.}
\label{CommutatorTransvection}
\end{center}
\end{figure}
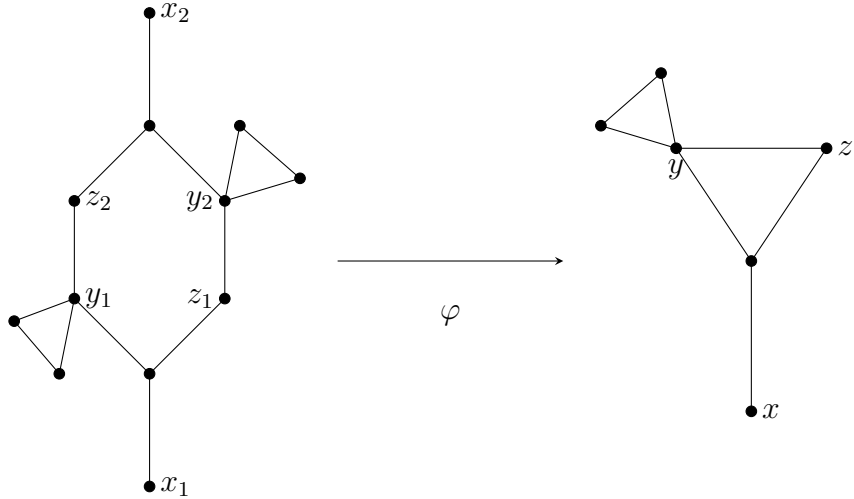

Except from the elements in $\Deck(\varphi)$, all other elements in $\FD(\varphi)$ that we listed so far act trivially on the first homology of $A_\Lambda$. We will prove that this is actually the case for the all group $\FD(\varphi)$. Before doing so, we need to establish a few lemmas and set up a few definitions.

Suppose that $\Gamma$ has no isolated vertices (and thus $\Lambda$ also has no isolated vertices). The identity of a RAAG will be denoted by $0$ to emphasize that it is an empty word.

\begin{lemma} \label{FarVertices}
Let $v_1,v_2 \in V\Gamma$ be two distinct vertices such that $d(v_1,v_2)\geq 3$. For an automorphism $f \in \Aut(A_\Gamma)$, if $v_1 \in \esupp(f(v_1))$, then $v_2 \notin \esupp(f(v_1))$.
\end{lemma}
\begin{proof}
By Proposition \ref{prop:essupp_auto}, we have $\esupp(f(v_1)) \subseteq V\Lambda_{\gtrsim v_1}$. Since $\Gamma$ has no isolated vertices, the distance between $v_1$ and any vertex in $V\Lambda_{\gtrsim v_1}$ is at most 2. Therefore, the fact that the distance between $v_1$ and $v_2$ is at least 3 proves the lemma.
\end{proof}

Let $\Lambda^+$ be the graph obtained from $\Lambda$ by adding an edge $\{x,y\}$ whenever $[\varphi(x), \varphi(y)]$ becomes trivial in $A_\Gamma$. In other words, $\{x,y\}$ is an edge in $\Lambda^+$ if either $\{\varphi(x),\varphi(y)\}$ is an edge of $\Gamma$ or $\varphi(x) = \varphi(y)$. 
The map $\varphi:\Lambda\to\Gamma$ can be extended to a map $\varphi':\Lambda^+\to\Gamma$ which also induces a surjective homomorphism $A_{\Lambda^+}\to A_\Gamma$.
We then have a map $\alpha \colon A_\Lambda \to A_{\Lambda^+}$ which is induced by the inclusion $\Lambda \to \Lambda^+$ and a map $\beta \colon A_{\Lambda^+} \to A_\Gamma$ which is induced by $\varphi'$. Moreover $\phi = \beta \circ \alpha$.

From now on, we say that two vertices $u$ and $u'$ in $\Lambda$ (or $\Lambda^+$) are \textit{deck-equivalent} if $\varphi(u) = \varphi(u')$ (since $\varphi$ is regular, $u$ and $u'$ are deck-equivalent if and only if there is a deck transformation $\sigma\in\Deck(\varphi)$ such that $\sigma(u)=u'$).
Let $x \in A_{\Lambda^+}$, and let $w=u_{1}^{\epsilon_1} \cdots u_{k}^{\epsilon_k}$ be a word representing $x$ where $u_{i} \in V\Lambda^+$ and $\epsilon_i \in \{1,-1\}$. An $\textit{exchange in w}$ is the operation consisting of replacing one of the $u_{i}^{\epsilon_i}$ by $\tilde{u}_{i}^{\epsilon_i}$ where $\tilde{u}_{i}$ is a vertex deck-equivalent to $u_{i}$.
\begin{figure}

\begin{center}
    \begin{tikzpicture}
    \filldraw[black] (4,0) circle (2pt) ;
    \filldraw[black] (4,2) circle (2pt) ;
    \filldraw[black] (6,2) circle (2pt) ;
    \filldraw[black] (6,0) circle (2pt) ;
    
    \draw (4,0) -- (4,2) -- (6,2) -- (6,0) -- (4,0);
    
    \filldraw[black] (1,1.6) circle (2pt) ;
    \filldraw[black] (1,0.4) circle (2pt) ;
    \filldraw[black] (0.2,-0.3) circle (2pt) ;
    \filldraw[black] (-1,-0.3) circle (2pt) ;
    \filldraw[black] (0.2,2.3) circle (2pt) ;
    \filldraw[black] (-1,2.3) circle (2pt) ;
    \filldraw[black] (-1.8,0.4) circle (2pt) ;
    \filldraw[black] (-1.8,1.6) circle (2pt) ;
    
    \draw (1,1.6) -- (1,0.4) -- (0.2,-0.3) -- (-1,-0.3) -- (-1.8,0.4) -- (-1.8,1.6) -- (-1,2.3) -- (0.2,2.3) -- (1,1.6);
    \draw (-3,0.6) node[above] {$\Lambda$} ;
    \draw (7,0.6) node[above] {$\Gamma$} ;
    \draw[-stealth](1.5,1) --node[above] {$\varphi$} (3.5,1) ;
    
    \filldraw[black] (4,-2.4) circle (2pt) ;
    \filldraw[black] (4,-3.6) circle (2pt) ;
    \filldraw[black] (3.2,-4.3) circle (2pt) ;
    \filldraw[black] (2,-4.3) circle (2pt) ;
    \filldraw[black] (3.2,-1.7) circle (2pt) ;
    \filldraw[black] (2,-1.7) circle (2pt) ;
    \filldraw[black] (1.2,-3.6) circle (2pt) ;
    \filldraw[black] (1.2,-2.4) circle (2pt) ;
    
    \draw (4,-2.4) -- (4,-3.6) -- (3.2,-4.3) -- (2,-4.3) -- (1.2,-3.6) --  (1.2,-2.4) -- (2,-1.7) -- (3.2,-1.7) -- (4,-2.4);
    \draw (4,-2.4) -- (2,-4.3);
    \draw (4,-2.4) -- (1.2,-3.6);
    \draw (4,-2.4) -- (1.2,-2.4);
    
    \draw (4,-3.6) -- (1.2,-3.6);
    \draw (4,-3.6) -- (2,-1.7);
    \draw (4,-3.6) -- (1.2,-2.4);
    
    \draw (3.2,-4.3) -- (1.2,-2.4);
    \draw (3.2,-4.3) -- (2,-1.7);
    \draw (3.2,-4.3) -- (3.2,-1.7);
    
    \draw (2,-4.3) -- (2,-1.7);
    \draw (2,-4.3) -- (3.2,-1.7);
    
    \draw (1.2,-3.6) -- (3.2,-1.7);
    \draw (0,-3.4) node[above] {$\Lambda^+$} ;
    
    \draw[-stealth](0.5,-0.6) -- (1.5,-1.5) ;
    \draw[-stealth](3.8,-1.5) --  (4.8,-0.6);
    
    \end{tikzpicture}
\caption{An example of a covering $\varphi \colon \Lambda \to \Gamma$ and of the associated graph $\Lambda^+$}
\end{center}

\end{figure}
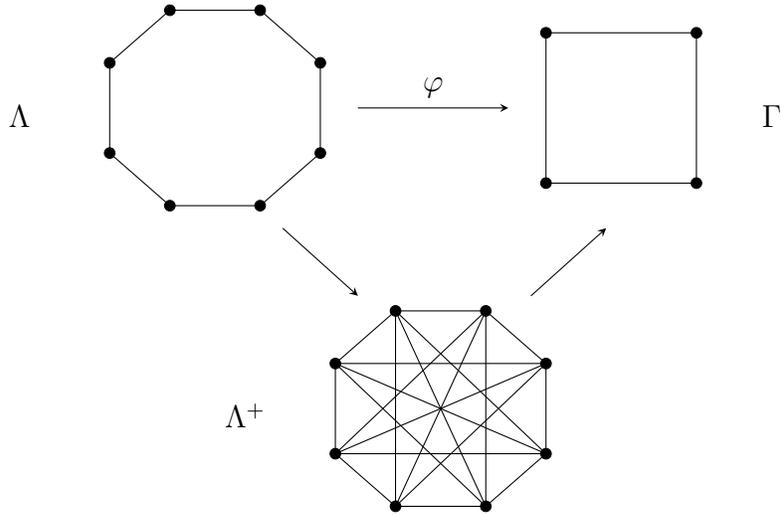

\begin{lemma} \label{ExchangeLemma}
Let $x \in A_{\Lambda^+}$ such that $\beta(x) = 0$, and let $w$ be a word representing $x$. Then, there exists a word $w'$, obtained from $w$ by a finite number of exchanges, such that $w'$ represents the trivial word in $A_{\Lambda^+}$.
\end{lemma}
\begin{proof}
Let $v,v' \in V\Gamma$. By the definition of $\Lambda^+$, we have that $[v,v'] = 0$ if and only if $[u, u'] = 0$ for all possible choices of vertices $u\in\varphi^{-1}(v)$ and $u'\in\varphi^{-1}(v')$. 
Write $w = u_{1}^{\epsilon_1} \cdots u_{k}^{\epsilon_k}$ where $u_{i} \in V\Lambda^+$ and $\epsilon_i \in \{1,-1\}$. Assume $\beta(w) =0$ but $w \neq 0$. But $\beta(w) =  \beta(u_{1}^{\epsilon_1}) \cdots \beta(u_{k}^{\epsilon_k}) = 0$. So it means that, after using commutations, some cancellations of the form $\beta(u_{i}^{\epsilon_i}) \beta(u_{m}^{\epsilon_m}) = 0$ occur. But every time such a cancellation appears, we can shuffle $w$ using the same commutations, to make $u_{i}^{\epsilon_i} u_{m}^{\epsilon_m}$ appear in $w$. But then, by replacing $u_{i}$, or $u_{m}$, by some deck-equivalent vertex $\tilde{u}_{i}$, or $\tilde{v}_{m}$, we can make the same cancellation appear. Since we can arrive to the trivial element from $\beta(u_{1}^{\epsilon_1}) \cdots \beta(u_{k}^{\epsilon_k})$ after finitely many such operations, the same is true for $w$.
\end{proof}

Denote the deck-equivalence classes of vertices of $\Lambda$ by $V_1, \cdots, V_k$. For each $V_i$, we can define a map $\Sigma_{V_i} \colon A_\Lambda \to \mathbb{Z}^n$ where $n$ is the degree of the covering map $\varphi$, or equivalently, the cardinality of any of the $V_i$. Let $V_i = \{u_{i_1}, \cdots, u_{i_n} \}$. 
The map  $\Sigma_{V_i}$ is then simply defined by $\Sigma_{V_i}(u_{i_j}) = (0, \cdots, 0 ,1, 0, \cdots,0)$ where the $1$ is at the $j$'th position and $\Sigma_{V_i}(u) = (0, \cdots,0)$ if $u \in V\Lambda\setminus V_i$. 
In other words, if $x \in A_\Lambda$ and $w$ is a word representing $x$, then $\Sigma_{V_i}$ counts all the signed occurrences of each $u \in V_i$ and records them in a vector. We then also have a map $\Sigma \colon A_\Lambda \to \mathbb{Z}^{nk}$ which is obtained by concatenating the images of all the $\Sigma_{V_i}$'s. We note in particular, that $\Sigma(x) = 0$ if and only if $\Sigma_{V_i}(x) = 0$ for all $i = 1,\cdots,k$. Finally, note that $\Sigma$ is nothing but the abelianization map and thus $\ker(\Sigma)$ is nothing but the commutator subgroup $(A_\Lambda)'$. In the following, we will abuse notations and also freely use the map $\Sigma$ and the maps $\Sigma_{V_i}$ on elements in $A_{\Lambda^+}$.

\begin{theorem} \label{CommensurableTorelli}
Suppose $\varphi:\Lambda\to\Gamma$ is a regular covering map of graphs without isolated vertices. Then the group $\FD(\varphi)$ is commensurable to a subgroup of $\IA(A_\Lambda)$. 
More precisely, $\FD(\varphi) / (\FD(\varphi) \cap \IA(A_\Lambda))$ is isomorphic to $\Deck(\varphi)$.
\end{theorem}
\begin{proof}
Let $F \in \FD(\varphi)$. By Corollary~\ref{lem:deck_lift}, there exists $\mu \in \Deck(\varphi)$ such that $F\mu$ is essential. For simplifying notation, we will thus now suppose that $F$ is essential for the rest of the proof. 

Our goal is to show that $\Sigma(u^{-1}F(u)) = 0$ for all $u \in V\Lambda$. Indeed, if that were to be true, that would mean that $u$ and $F(u)$ represent the same element in $H_1(A_\Lambda)$ for each $u \in V\Lambda$, so that $F$ indeed acts trivially on $H_1(A_\Lambda)$.

By abuse of notation, we consider $F(u)$ as an element of $A_\Lambda$ and as a reduced word at the same time. Let $x(u)$ be a word such that $c(u) = x(u)^{-1} F(u) x(u)$ is a cyclically reduced word. Then,
\begin{enumerate}
    \item $u \in \supp(c(u))$
    \item\label{123} if $w\in \supp(c(u))$, then $w' \notin \supp(c(u))$ for any $w'$ deck-equivalent to $w$.
\end{enumerate}
Note that item \noindent\eqref{123} holds by Lemma \ref{FarVertices} and the fact that any two distinct vertices in the inverse image of a vertex under a covering map of graphs are at distance $\geq 3$.
Let $K$ be $\ker(\phi \colon A_\Lambda \to A_\Gamma)$. Since $F$ is a lift of the identity, we have that $u^{-1}F(u) \in K$. Since $F(u)$ can be represented by the word $x(u) c(u) x^{-1}(u)$, the word $w = u^{-1} x(u) c(u) x^{-1}(u)$ represents $u^{-1}F(u)$. 
Note that $w$ is also a word representing $\alpha(u^{-1}F(u))$ in $A_{\Lambda^+}$ via the map $\alpha:\Lambda\rightarrow\Lambda^+$, and $\beta(w)=0$.

Without loss of generality, suppose that $u \in V_1$, where $V_1, \cdots,V_k$ is the ordered set of deck-equivalent vertices of $\Lambda$ (for each $V_i$, we can give any order on vertices). For simplicity, we also suppose that $u$ is the first element in $V_1$. We want this to show that $\Sigma_{V_i}(w) = 0$ for all $i =1,\cdots, k$.

\begin{enumerate}
    \item[(I)]\label{I} Suppose $i \neq 1$. Then $$\Sigma_{V_i}(u^{-1} x(u) c(u) x^{-1}(u)) = \Sigma_{V_i}( x(u) c(u) x^{-1}(u))= \Sigma_{V_i}(c(u)).$$
    If $\Sigma_{V_i}(c(u)) = 0$, then $\Sigma_{V_i}(w)=0$. Otherwise, item \eqref{123} implies that $\Sigma_{V_i}(c(u))= (0,\cdots,0,m,0,\cdots,0)$ for some integer $m \neq 0$. But, by Lemma \ref{ExchangeLemma}, we can modify the word $w$ by a finite sequence of exchanges in order to get to the trivial word. But if $w'$ is a word obtained from $w$ by a finite number of exchanges, the sum of the coordinates of $\Sigma_{V_i}(w)$ is equal to the sum of the coordinates of $\Sigma_{V_i}(w')$. Indeed, an exchange is only replacing a $u_i^{\epsilon}$ by $u_j^{\epsilon}$ where $u_i$ and $u_j$ are deck-equivalent, and both of $u_i^{\epsilon}$ and $u_j^{\epsilon}$ add $\epsilon$ to the sum of the coordinates of $\Sigma_{V_i}$.  Hence, that sum must remain equal to $m$, which contradicts the fact that we can reduce $w$ to the trivial word. So $\Sigma_{V_i}(w) = 0$.
    \item[(II)] Suppose $i = 1$. Since $u \in \supp (c(u))$, we have $$\Sigma_{V_1}(u^{-1}x(u) c(u) x^{-1}(u)) = \Sigma_{V_1} (u^{-1} c(u)).$$ But then, $ \Sigma_{V_1} (u^{-1} c(u)) = (m,0,\cdots,0)$. The same argument used in case \noindent\eqref{I} shows that $\Sigma_{V_1}(w) = 0$.
\end{enumerate}
This thus shows that $\Sigma (w) = 0$ and concludes the proof.
\end{proof}

As a corollary of Theorem \ref{CommensurableTorelli}, we can obtain a short exact sequence in the spirit of the one coming from the Birman-Hilden theory for surfaces.

\begin{corollary} \label{cor:Birman-Hilden}
For the regular covering map $\varphi: \Lambda \to \Gamma$ given in Theorem~\ref{CommensurableTorelli}, the following short exact sequences are exact:
$$1 \to \FD(\varphi) \to \FAut(\varphi) \to \LAut(\varphi) \to 1$$ and $$1 \to \Deck(\varphi) \to \FAut(\varphi)/(\FAut(\varphi) \cap \IA(A_\Lambda)) \to \LAut(\varphi)/(\LAut(\varphi) \cap \IA(A_\Gamma)) \to 1$$
\end{corollary}
\begin{proof}




Let $q \colon \FAut(\varphi) \to \LAut(\varphi)$ be the quotient map. Recall that if $G$ is a group and $N$ is a normal subgroup of $G$, then $GN = \{g \cdot n \mid g \in G, n \in N \}$ is a subgroup of $G$. We claim that there is a well defined induced surjective homomorphism 
$$
q_* \colon \dfrac{\FAut(\varphi) \IA(A_\Lambda)}{\IA(A_\Lambda)} \to \dfrac{\LAut(\varphi) \IA(A_\Gamma)}{\IA(A_\Gamma)}.
$$
This homomorphism $q_*$ is simply defined by setting, for $F \in \FAut(\varphi)$ and $I \in \IA(A_\Lambda)$ $q_*((F \cdot I)\IA(A_\Lambda)) := (q(F)\cdot e) \IA(A_\Gamma)$, where $e$ is the identity element in $\IA(A_\Gamma)$.
Since $(F \cdot I)(F' \cdot I') = (FF' \cdot ((F')^{-1}IF'I')$ and $\IA(A_\Lambda)$ is a normal subgroup, the map $q_*$ is indeed a homomorphism. Then, it is an easy exercise to check that the following sequence is exact

$$
1 \to \dfrac{\FD(\varphi) \IA(A_\Lambda) }{\IA(A_\Lambda)} \to \dfrac{\FAut(\varphi) \IA(A_\Lambda)}{ \IA(A_\Lambda)}\to \dfrac{\LAut(\varphi) \IA(A_\Gamma)}{\IA(A_\Gamma)} \to 1
$$
Therefore, by applying the second isomorphism theorem for groups, and the last part of Theorem \ref{CommensurableTorelli}, we obtain the desired short exact sequence of the statement of this corollary.
\end{proof}

\section{Lifting of transvections revisited}\label{sec:lift_transv}

In this section, we come back to the combinatorial lifting criterion for transvections of Lemma \ref{lem:lift_transvection} and show that it is in fact a strict condition for liftability when $\Gamma$ has no isolated vertices.

Suppose $\varphi:\Lambda\rightarrow\Gamma$ is a regular covering map of graphs without isolated vertices. Let $m$ and $n$ be the number of vertices of $\Gamma$ and the degree of the covering map $\varphi$, respectively.
Let $u$ and $u'$ be two vertices in $\Lambda$ which are not comparable under the link-star order. 
Lemma \ref{OrderOnV} implies that $V\Lambda$ is an ordered set but this order is not uniquely determined. Moreover, an order on $V\Lambda$ can be suitably chosen so that $u$ comes before $u'$. 

\begin{lemma} \label{ordering1}
It is possible to label the vertices of $V\Lambda = \{u_1, \cdots, u_{mn}\}$ in such a way that 
\begin{enumerate}
    \item if $u_i \gnsim u_j$, then $i < j$
    \item if $u = u_i$ and  $u' = u_j$ in this ordering, then $i < j$.
\end{enumerate}
\end{lemma}

\begin{proof}
Let $V\Lambda_{\lesssim u'}$ be the set of vertices that are smaller than or equivalent to $u$ in the link-star order sense. Then, we can order the vertices of $V\Lambda - V\Lambda_{\lesssim u'}$ in such a way that $(1)$ holds by applying Lemma \ref{OrderOnV}. Now, we do the same thing for ordering the vertices of $V\Lambda_{\lesssim u'}$. Concatenating these two orders gives us a desired ordering of $V\Lambda$.
\end{proof}

Recall that we say that two vertices $u_i$ and $u_j$ of $\Lambda$ are deck-equivalent, if there exists $\sigma \in \Deck(\varphi)$ such that $u_j = \sigma(u_i)$.

\begin{lemma} \label{ordering2}
It is possible to label the vertices of $V\Lambda = \{u_1, \cdots, u_{mn}\}$ in such a way that 
\begin{enumerate}
    \item if $u_i \gnsim u_j$, then $i < j$,
    \item if $u = u_i$ and  $u' = u_j$ in this ordering, then $i < j$, and
    \item if $u_i$ and $u_j$ are deck-equivalent, with $i \leq j$, then every vertex $u_k$ such that $i \leq k \leq j $ is also deck-equivalent to $u_i$ and $u_j$.
\end{enumerate}
\end{lemma}

\begin{proof}
We make the set of deck-equivalence classes of vertices of $\Lambda$ into a pre-ordered set $E$ by setting $[x] \gtrsim [y]$ if and only if for some vertex $u_i \in [x]$ and $u_j \in [y]$ we have $u_i \gtrsim u_j$. Then, we order the set $E$ as in Lemma \ref{ordering1}, in such a way that the equivalence class corresponding to $u$ comes before the one corresponding to $u'$. Now, we just blow up each equivalence class to obtain an ordering of $V\Lambda$ that has all the desired properties. Note that this works because the symmetries of $\Lambda$ preserve the link-star order, so that if two vertices are deck-equivalent, they are either equivalent or non comparable in the link-star order.
\end{proof}



Let $F \in \FAut(\varphi)$ be a lift of $f \in \LAut(\varphi)$. We investigate the relation between the action of $f$ on $H_1(A_\Gamma)$ and the action of $F$ on $H_1(A_\Lambda)$. 
In order to denote the action of an automorphism on the first homology by matrices, we need to fix an order on the generating set. Thus, we fix an order on $V\Lambda$ obtained from Lemma \ref{ordering2} so that the deck-equivalence classes of $V\Lambda$ in this order are denoted by $V_1,\cdots,V_m$. 
We also fix the order on $V\Gamma$ which is induced from $V\Lambda$ via the covering map $\varphi$. Note that the induced order on $V\Gamma$ may not satisfy the condition in Lemma \ref{OrderOnV}.

Let $M$ be an $m \times m$ square matrix and let $\tilde{M}$ be an $nm \times nm$ square block matrix with block of size $n\times n$. We say that $\tilde{M}$ is a \textit{blow up} of $M$ if for every $i,j = 1, \cdots, m$, the $(i,j)$ block of $\tilde{M}$ is a matrix in which every column consists only of 0 entries except for one entry which is equal to $M(i,j)$.

The remaining facts of this section are based on Lemma \ref{FarVertices}, and thus, from now on, we assume that $\Gamma$ (and also $\Lambda$) has no isolated vertices.

\begin{lemma}
Let $F \in \FAut(\varphi)$ be a lift of $f \in \LAut(\varphi)$. Then, under an appropriate ordering of the vertices of $\Lambda$ and $\Gamma$, the matrix $\tilde{M}$ corresponding to the action of $F $ on $H_1(A_\Lambda)$ is a blow up of the matrix $M$ corresponding to the action of $f$ on $H_1(A_\Gamma)$. 
\end{lemma}

\begin{proof}
Let $v$ be a vertex in $\Gamma$ and let the abelianization of $f(v)$ be the vector $(x_1,\cdots,x_m)$. Then, for any vertex $u \in \varphi^{-1}(v)$, we have that $\Sigma_{V_i}(F(u))$ is a vector such that all of its coordinates are 0's expect for one coordinate which is equal to $x_i$ (note that the size of the vector $\Sigma_{V_i}(F(u))$ is equal to the degree of the covering map $\varphi$ which may not be equal to $m$). Indeed, as in Theorem \ref{CommensurableTorelli}, we can show that $\Sigma(F(u))$ has only one non-zero component. Moreover, if $w$ is any word in $A_\Lambda$ such that $\phi(w) = f(v)$, then $w^{-1} F(u) \in \ker(\phi)$. Using Lemma \ref{ExchangeLemma}, we show, the same way as we did in Theorem \ref{CommensurableTorelli}, that $\Sigma_{V_i}(F(u)) = \Sigma_{V_i} (w)$.
\end{proof}

We finally can prove that there are no other liftable transvections than the one described by Lemma \ref{lem:lift_transvection}.
\begin{corollary}\label{lem:lift_transvection2}
For a regular covering map $\varphi: \Lambda \to \Gamma$ of graphs without isolated vertices, let $v$ and $v'$ be vertices in $\Gamma$ satisfying $\lk(v) \subseteq \st(v')$.
Then a transvection $T_v^{v'}$ is liftable if and only if $v \lesssim_\varphi v'$.
\end{corollary}

\begin{proof}
We only need to prove the "only if" direction. Assume that $v \lesssim v'$ and there exists a lift $F$ of $T_v^{v'}$, but $v$ and $v'$ are not comparable under the suborder $\lesssim_\varphi$. By Lemma \ref{lem:deck_lift}, there exists a deck transformation $\mu \in \Deck(\varphi)$ such that $F\mu$ is essential, and as in the proof of Proposition \ref{prop:decom}, $F\mu$ can be written as a product of inversions, partial conjugations and transvections. We order the vertices of $\Lambda$ as in Lemma \ref{ordering2}. Then, the matrix $\tilde{N}$ corresponding to the action of $F\mu$ on  $H_1(A_\Lambda)$ is a block upper triangular matrix where the blocks correspond to the the deck-equivalence classes of vertices. Indeed, partial conjugations act trivially on $H_1(A_\Lambda)$, inversions act as diagonal matrices and transvections acts as upper triangular matrices due the the choice of the ordering. We then deduce that the matrix $\tilde{M}$ corresponding to the action of $F$ on $H_1(A_\Lambda)$ is also block upper triangular, since it is obtained by multiplying $\tilde{N}$ by a block triangular matrix corresponding to the action of $\sigma^{-1}$.

On the other hand, we know that the matrix $\tilde{M}$ is a blow up of the matrix corresponding to the action of $T_v^{v'}$ on $H_1(A_\Gamma)$. Choose any vertices $u\in\varphi^{-1}(v)$ and $u'\in\varphi^{-1}(v')$. 
By our assumption, $u$ and $u'$ are not comparable under the link-star order. It means that we can change the order on $V\Lambda$ without ruining the proof such that $u$ comes before $u'$.
Then, the column of $\tilde{M}$ corresponding to $u$ contains exactly two non zeroes entries, with value 1. One is in the equivalence class of $u$ and the other one is in the equivalence class of $u'$. But then, in the ordering of the vertices that we chose, this matrix would not be block upper triangular since $u$ comes before $u'$ in the revised order and they are not in the same deck-equivalence class.
\end{proof}

\bibliographystyle{amsplain}
\bibliography{KOST_references}

\end{document}